\documentclass[12pt,%reqno
]{amsart}
\usepackage{amssymb,mathrsfs}
\usepackage{color,umoline}
\usepackage[dvipsnames]{xcolor}
\usepackage{graphicx}
\usepackage{enumitem}
\usepackage{mathabx} % -> can use the symbol 'bigboxvoid'

\def\norm#1{\|#1\|}

\def\wh#1{\widehat{#1}}

\newcommand{\F}{\mathcal{F}}

\newcommand{\cS}{\mathcal{S}}

\newcommand{\T}{\mathcal{T}}
\newcommand{\K}{\mathcal{K}}

\newcommand{\C}{\mathbb{C}}

\newcommand{\R}{\mathbb{R}}
\newcommand{\Z}{\mathbb{Z}}

\newcommand{\be}{\beta}

\newcommand{\e}{\varepsilon}

\newcommand{\la}{\lambda}

\newcommand{\re}{\mathop{\mathrm{Re}}}

\newcommand{\supp}{\operatorname{supp}}

\newcommand{\Del}[1]{}

\newcommand{\finv}[1]{\mathcal F^{-1}\big({#1}\big)}

\numberwithin{equation}{section}

\newtheorem{thm}{Theorem}[section]
\newtheorem{cor}[thm]{Corollary}
\newtheorem{lem}[thm]{Lemma}
\newtheorem{prop}[thm]{Proposition}

\theoremstyle{remark}
\newtheorem{rem}{Remark}

\usepackage[pdftex]{hyperref}

\begin{document}

\title[Uniform Sobolev inequalities]{Uniform Sobolev inequalities for second order non-elliptic  differential operators}

\author[E. Jeong]{Eunhee Jeong}
\author[Y. Kwon]{Yehyun Kwon}
\author[S. Lee]{Sanghyuk Lee}

\address{Department of Mathematical Sciences, Seoul National University, Seoul 151-747, Republic of Korea}
\email{moonshine10@snu.ac.kr}
\email{kwonyh27@snu.ac.kr}
\email{shklee@snu.ac.kr}

%\thanks{This work was supported by NRF of Korea (grant no. 2015R1A2A2A05000956). }

\subjclass[2010]{35B45, 42B15} \keywords{Sobolev inequality, uniform
estimate, non-elliptic}

\begin{abstract}
We study  uniform Sobolev inequalities for the  second order
differential operators $P(D)$ of non-elliptic type. For $d\ge3$ we
prove that the Sobolev type estimate $\|u\|_{L^q(\mathbb{R}^d)}\le C
\|P(D)u\|_{L^p(\mathbb{R}^d)}$ holds with $C$ independent of   the
first order and the constant  terms of $P(D)$ if and only if
$1/p-1/q=2/d$ and $\frac{2d(d-1)}{d^2+2d-4}<p<\frac{2(d-1)}d$. We
also obtain restricted weak type endpoint estimates for the critical
$(p,q)=(\frac{2(d-1)}{d},\frac{2d(d-1)}{(d-2)^2})$,
$(\frac{2d(d-1)}{d^2+2d-4}, \frac{2(d-1)}{d-2})$. As a consequence,
the result extends  the class of functions  for which the unique
continuation for the inequality $|P(D)u|\le|Vu|$ holds.
\end{abstract}

\maketitle

\section{Introduction}
Let $Q$ be a non-degenerate real quadratic form defined on $\mathbb{R}^d$, $d\ge 3$, which is given by
\begin{equation}\label{quad}
Q(\xi)=-\xi_1^2-\cdots -\xi_k^2+\xi_{k+1}^2 +\cdots +\xi_d^2 ,
\end{equation}
where $1\le k\le d$. We consider the constant coefficient second order differential operator
\begin{equation*}
P(D)=Q(D)+\sum_{j=1}^{d}a_jD_j +b,
\end{equation*}
where $D=(D_1,\cdots ,D_d)$, $D_j=\frac{1}{2\pi i}\frac{\partial}{\partial x_j}$ and  $a_1,\cdots,a_d$, $b$ are complex numbers.
We call $P$ `elliptic' if $k=d$ and `non-elliptic' otherwise.

The Sobolev type estimate
\begin{equation}\label{Sob}
\|u\|_{L^q(\mathbb{R}^d)} \le C\|P(D)u\|_{L^p(\mathbb{R}^d)}
\end{equation}
which holds for  $u\in W^{2,p}(\mathbb{R}^d)$
has been of interest in connection to studies of partial differential equations. Here the function space $W^{2,p}(\mathbb{R}^d)$ denotes the second order $L^p$-Sobolev space. If $P(D)=\frac{1}{4\pi^2}\Delta$, \eqref{Sob} is a particular case of  the classical Hardy-Littlewood-Sobolev inequality. When $P(D)$ is non-elliptic, \eqref{Sob}
is closely related to the inhomogeneous Strichartz estimates (\cite{KT, fos, vil, ls, Ta}) for the dispersive equations such as the wave and the Klein-Gordon equations {(see \cite{St, msw, Na-Oz}). For these equations,  estimates \eqref{Sob} were first shown by Strichartz \cite{St} for some $p,q$.

On the other hand, related to a type of  Carleman estimate (e.g. see \eqref{cman}) which is used in the study of unique continuation,
the estimate \eqref{Sob} with $C$ independent of the first and zero order parts of $P(D)$ has been studied.
For such an estimate to hold, by scaling  it is necessary  that   the condition
\begin{equation}\label{scaling}
\frac1p-\frac1q=\frac2d
\end{equation}
holds.  For the elliptic $P(D)$, Kenig, Ruiz, and Sogge \cite{KRS}
characterized the optimal range of the Lebesgue exponents $p$ and
$q$ for which the uniform Sobolev inequality \eqref{Sob} holds. More
precisely, they showed  that the uniform estimates \eqref{Sob} are
true if and only if $1/p-1/q=2/d$ and
$2d/(d+3)<p<2d/(d+1)$\footnote{For those pairs of $p,q$, $(1/p,1/q)$
is in the open line segment $AA'$ in Figure \ref{fig}.}. For
non-elliptic $P(D)$,  it was shown (\cite[Theorem 2.1]{KRS}) that
the uniform Sobolev inequality \eqref{Sob} is true provided
$1/p+1/q=1$ and $1/p-1/q=2/d$, i.e., $(p,q)=(2d/(d+2),2d/(d-2))$
(the point $F$ in Figure \ref{fig}).

However it seems natural to expect that  the uniform bounds \eqref{Sob} continue to hold for $(p, q)$ other than $(2d/(d+2),2d/(d-2))$.
No such estimate seems to be established before (see {\it Remark} \ref{remark1} below Theorem \ref{uniSob1}).  A computation shows that
 in addition to \eqref{scaling}  the condition
 \begin{equation}\label{cond2}
{p} <\frac{2(d-1)}{d}\, ,\quad
\frac{2(d-1)}{d-2}<q\,
\end{equation}
should be satisfied.
(See Section \ref{nec-sec}.)

In this paper  we consider the uniform estimate \eqref{Sob} for non-elliptic $P(D)$ ($1\le k\le d-1$)  and
 extend the previous results in \cite{KRS} to the optimal range of exponents $p$ and $q$. Hence we completely characterize the range of
 $p,q$  for which the uniform estimate \eqref{Sob} holds.
 More precisely, we shall prove the following which is our main theorem.

\begin{thm} \label{mainthm}
Let $d\ge 3$ and $P(D)$  be a non-elliptic second order differential operator with constant coefficients. Then there exists an absolute constant $C$, depending only on $d,k,p$ and $q$, such that  \eqref{Sob} holds uniformly in $a_1,\dots, a_d,$ $b$, if and only if $(p,q)$ satisfies \eqref{scaling} and \eqref{cond2}\,\footnote{This pair $(1/p,1/q)$ lies on the open line segment $BB'$ in Figure \ref{fig}.}.
Furthermore,  if  $(p,q)$ is either
$
\big(\frac{2(d-1)}{d},\frac{2d(d-1)}{(d-2)^2}\big) $  or $\big(\frac{2d(d-1)}{d^2+2d-4},\frac{2(d-1)}{d-2}\big)
$
\footnote{These correspond to the points $B$ and $B'$ in Figure \ref{fig}.},  we have the  restricted weak type bound
\begin{equation}\label{wk}
\|u\|_{q,\infty} \le C\|P(D)u\|_{p,1}.
\end{equation}
\end{thm}

The argument in \cite{KRS} which shows \eqref{Sob} for $1/p+1/q=1$ is  based on interpolation along a complex analytic family of distributions (see \cite{Stein}) for which  $L^1$-$L^\infty$ and $L^2$-$L^2$ estimates are relatively easier to obtain. Since this type of argument  heavily relies  on the structure of the specific family of distributions, the  method is  less flexible and seems restrictive. Instead, we directly analyze the associated  multiplier operators of which singularity lies on the surface given by the function $Q$. For this purpose,  we follow the  approach which is rather typical in the study of boundedness of operators of  Bochner-Riesz types \cite{CKLS, L, ls0}. In fact, we dyadically  decompose  the multiplier operator away from  the singularity by taking into account  the distance to the surface. This gives multiplier operators of different scales which are less singular and for these operators various $L^p$-$L^q$ estimates become available.  However, in order to prove the desired estimates  we need to obtain the sharp bounds in terms of the distance to the singularity (for example, see the estimates \eqref{psi1}, \eqref{psi2}). For this purpose we decompose the multiplier operator by imposing additional cancellation property so that  the resulting operators have the correct $L^1$-$L^\infty$ bound  (see Section \ref{tdel} for  details).\\

\begin{figure}\label{fig}
\includegraphics[width=\linewidth]{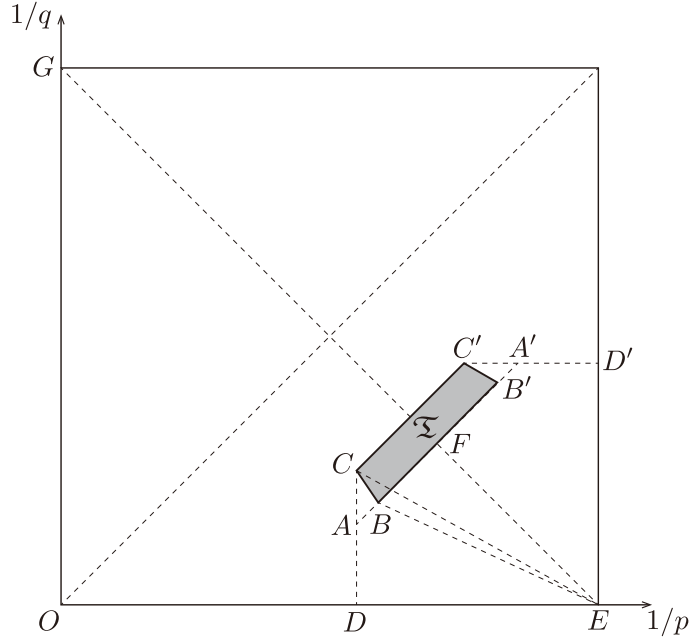}
\caption{The points $A=(\frac{d+1}{2d},\frac{d-3}{2d})$, $B$, $C$, $D=(\frac{d+1}{2d},0)$, $E=(1,0)$, $F=(\frac{d+2}{2d},\frac{d-2}{2d})$, $G=(0,1)$, $O=(0,0)$, and the dual points $A'$, $B'$, $C'$, $D'$ when $d\ge3$. The line segments $AA'$, $CC'$, $BE$, and $CE$ are on the lines $\frac{1}{p}-\frac{1}{q}=\frac{2}{d}$, $\frac{1}{p}-\frac{1}{q}=\frac{2}{d+1}$, $\frac{1}{q}=\frac{d-2}{d}(1-\frac{1}{p})$, and $\frac{1}{q}=\frac{d-1}{d+1}(1-\frac{1}{p})$, respectively.}
\end{figure}

\noindent\textit{Uniform resolvent estimate.} By the reduction in \cite{KRS} the crucial step for the proof of Theorem \ref{mainthm} is to obtain the  uniform resolvent estimate
\begin{equation}\label{resol}
\|u\|_{L^q(\mathbb{R}^d)}\le C\|(Q(D)+z)u\|_{L^p(\mathbb{R}^d)},\quad z\in\mathbb{C}
\end{equation}
for $ u\in W^{2,p}(\mathbb{R}^d).$  When $Q(D)=\frac{1}{4\pi^2}\Delta$, in \cite{KRS}  the resolvent estimates \eqref{resol} were proved for all $p$ and $q$ satisfying the conditions $1/p-1/q=2/d$ and $2d/(d+3)<p<2d/(d+1)$, by making use  of  the oscillatory integral estimate due to  Stein \cite{S}. From these estimates the  uniform inequalities \eqref{Sob} were obtained in  the optimal range of $p$, $q$. These correspond to the open line segment $AA'$ in Figure \ref{fig}. In particular, if $z$ is a positive real number, the estimate is related to Bochner-Riesz operator of order $-1$.  The interested reader is referred to \cite{Bo, b,bmo,  Gu, CKLS}.

Also, when $Q(D)$ is non-elliptic, Kenig, Ruiz, and Sogge proved that the  uniform resolvent estimate
\begin{equation}\label{resolne}
\|u\|_{L^q(\mathbb{R}^d)}\le C\|(Q(D)+z)u\|_{L^p(\mathbb{R}^d)},\quad |z|\ge1
\end{equation}
is true whenever $1/p+1/q=1$ and $2/(d+1)\le1/p-1/q\le2/d$ \cite[Theorem 2.3]{KRS}. If $1/p-1/q=2/d$, the uniform estimate \eqref{resolne} is equivalent to \eqref{resol} by scaling.

\medskip

In what follows we extend the known range of $p,q$ for which \eqref{resolne}} holds. In order to state  our result  we set
\[ B=\Big(\frac{d}{2(d-1)},\frac{(d-2)^2}{2d(d-1)}\Big),\  C=\Big(\frac{d+1}{2d},\frac{(d-1)^2}{2d(d+1)}\Big), \]
and also define $B'$ and $C'$ by  setting $P'=(1-y, 1-x)$ for $P=(x,y)$ (see Figure \ref{fig}). Let us denote by $\mathfrak T$  the closed trapezoid with vertices $B, B', C, C'$ from which the points $B, B', C, C'$ are removed.

\begin{thm}\label{uniSob1} Let $d\ge3$ and let $Q$ be a non-elliptic, non-degenerate real quadratic form as in Theorem \ref{mainthm}.
Let $(1/p,1/q)\in \mathfrak T$. Then there is an absolute constant $C$ such that \eqref{resolne} holds with $C$ independent of  $z$, $|z|\ge 1$. Furthermore, if $(1/p, 1/q)$ is one of the vertices $B, B', C, C'$,
then we have $L^{p,1}$-$L^{q,\infty}$ estimate.
\end{thm}

\begin{rem}\label{remark1} It was claimed in \cite{bmo} (Theorem $6'$) that the $L^p$-$L^q$ estimates in Theorem \ref{mainthm} were established by combining the interpolation method (Theorem $1^\prime$) in \cite{bmo} and the estimates for the analytic family which are used in
\cite{KRS}. But the argument there does not seem to work.  In fact, to show  \eqref{resolne} by following the lines of argument in \cite{bmo} (see p.164) one has to consider the analytic family of operators $\{T_\lambda\}_{\lambda\in \mathbb C}$ which is defined along parameter $\lambda$  by
\[\widehat{T_\lambda f}(\xi)=C_\lambda(Q(\xi)+z)^\lambda \widehat f(\xi)\] with a suitable complex number  $C_\lambda$ (see \cite{KRS, bmo}).  But the crucial assumption  $|T_\lambda^*T_\lambda f| \le C|T_{2\re \lambda} f|$  of Theorem $1^\prime$  is not valid for $ T_\lambda$.
This inequality can not be satisfied  for  general complex number $z$ unless $z$ is real because
\[T_\lambda^*T_\lambda f=|C_\lambda|^2 \finv{  (Q(\xi)+\bar z)^{\bar \lambda} (Q(\xi)+z)^\lambda \widehat f(\xi)}.\]
\end{rem}

\

% \subsubsection*{Restriction-extension operator}
\noindent\textit{Restriction-extension operator.} The uniform estimates \eqref{Sob} and \eqref{resolne} are  closely related to
 the $L^2$-Fourier restriction estimate to the surfaces $\Sigma_\rho=\{\xi: Q(\xi)=\rho\}$.  We note that
  \[\frac1{Q(\xi)\pm 1+i\epsilon}-\frac1{Q(\xi)\pm1-i\epsilon}
 =\frac{-2i\epsilon}{(Q(\xi)\pm 1)^2+\epsilon^2}\to -2\pi i\,\delta(Q(\xi)\pm 1)\]
as $\epsilon\to 0$ in the sense of tempered distribution. Here $\delta$
is the delta distribution and  $\delta(Q(\xi)\pm \rho)$ is the composition  of the distribution $\delta$ with the smooth function $Q(\xi)\pm \rho$.
For $\rho\neq 0$, $\delta(Q(\xi)-\rho)$ is well defined. See \cite[pp.133--137] {H} for detail. It should be noted that  $\delta(Q(\xi)-\rho)$ coincides with the canonical measure on $\Sigma_\rho$.  Hence, the uniform estimate
\eqref{Sob} (also \eqref{resol} and \eqref{resolne})  implies
\begin{equation}\label{TT*0}
\Big\| \int \delta(Q(\xi)\pm 1) e^{2\pi ix\cdot \xi} \widehat{f}(\xi) d\xi  \Big\|_{L^q(\mathbb{R}^d)} \le C \| f \|_{L^p(\mathbb{R}^d)}, \quad f\in \mathcal{S}(\mathbb{R}^d).
\end{equation}
(Here $\mathcal{S}(\mathbb{R}^d)$ denotes the Schwartz space.)  Instead of the term {\it extension operator} which is typically used and somehow misleading   we call the operator $f\to \finv{\delta(Q\pm 1) \widehat{f}\,\,}$ {\it restriction-extension operator} since it is  composition of the Fourier restriction and extension (its dual) operators defined by the surface $\Sigma_{\mp 1}$.    As is clear to experts, \eqref{TT*0} is closely related to the inhomogeneous Strichartz estimates. See  \cite{KT, fos, vil,  Ta} and references therein.  
Especially, if $Q(\xi)=-\xi_1^2+ \xi_2^2+\dots+ \xi_d^2$, \eqref{TT*0} relates to the estimates for the Klein-Gordon equation. For example, see \cite{msw, Na-Oz} for earlier results.

   By scaling \eqref{TT*0} implies the estimate
\begin{equation}\label{TT*00}
\Big\| \int \delta(Q(\xi)-\rho) e^{2\pi ix\cdot \xi} \widehat{f}(\xi) d\xi  \Big\|_{L^q(\mathbb{R}^d)} \le C { |\rho|^{\frac 1 2 (\frac d p -\frac d q -2) }}  \| f \|_{L^p(\mathbb{R}^d)}, \quad f\in \mathcal{S}(\mathbb{R}^d),
\end{equation}
for $\rho\ne0$. This estimate will play an important role in proving {\eqref{resolne}}. 
Even if \eqref{TT*0} is obviously weaker than \eqref{resolne}, in view of our argument which  proves
\eqref{resolne}   the estimate \eqref{TT*0}  may be considered to be almost as strong as \eqref{resolne}.
In Section \ref{res-ex}  we show that \eqref{TT*0} holds for the same $p, q$ as in Theorem  \ref{uniSob1} (see Proposition \ref{restriction}). 

\

The rest of this paper is organized as follows. In section 2 we state and prove technical lemmas which decompose the delta and principal value distributions  into a sum of functions while these functions possess favorable cancellation properties.  These lemmas will be crucial for obtaining the sharp estimates.  Also, we show sharp estimates for the multiplier operators associated with  the surfaces $\Sigma_\rho$. In section 3 we prove the restriction-extension estimate \eqref{TT*0} and investigate its necessary conditions, which in turn give the optimality of the range of $p,q$ in Theorem \ref{mainthm}. In section 4 we prove Theorem \ref{mainthm} and Theorem \ref{uniSob1}. In section 5, as applications, we shall briefly mention results on Carleman inequalities and unique continuation.\\

%\subsubsection*{Notations} 
\noindent\textit{Notations.} Throughout this paper  the constant  $C$ may vary line to line. For $A,B>0$ we write $A\lesssim B$ to denote $A\le C B$ for some constant $C>0$ independent of $A,B$. By $A\sim B$ we mean $A\lesssim B$ and $B\lesssim A$. Also, $\widehat{f}$ and $f^\vee$ denote the Fourier and inverse Fourier transforms of $f$, respectively;
\[
\widehat{f}(\xi)=\int_{\mathbb{R}^d}e^{-2\pi i x\cdot\xi} f(x)dx , \quad f^\vee(x)=\int_{\mathbb{R}^d}e^{2\pi i x\cdot\xi}f(\xi)d\xi.
\]
We also use the notations  $\mathcal F(f)$ and $\finv{f}$  for the Fourier and the inverse Fourier transforms of $f$, respectively. In the sequel we frequently need to consider points $x,\eta \in \mathbb{R}^d$ in separated variables.  We write $x=(x_1, x', x'', x_d)\in\R\times \R^{k-1}\times\R^{d-k-1}\times \R$ and $\eta=(\eta_1,\eta',\eta'',\eta_d)\in\R\times \R^{k-1}\times\R^{d-k-1}\times \R$. We also write as {$\tilde{x}=(x_1,x',x'')\in\R^{d-1}$} and $\tilde{\eta}=(\eta_1,{\eta'},{\eta''})\in\R^{d-1}$.

\section{Preliminaries}

\subsection{Decomposition of distributions}
We now state and prove the following lemmas which provide dyadic decompositions of the delta and the principal value distributions. 
These are to be used in Section 3.

\begin{lem}\label{del}
There is a function $\psi\in \cS(\R)$ of which  Fourier transform  $\widehat{\psi}$ is supported in  $[-2,-1/2]\cup[1/2,2]$ such that, for all $g\in \cS(\R)$,
\begin{align*}
g(0)=\sum_{j=-\infty}^\infty2^{-j}\int \psi(2^{-j}x)g(x)dx.
\end{align*}

\begin{proof} The proof of this lemma is rather straightforward.  Let $\phi$ be a smooth function supported in $[-2,-1/2]\cup[1/2,2]$ such that
$\sum_{j=-\infty}^\infty \phi(2^jx)=1$ for $x\neq 0$. Then, for $g\in \cS(\R)$
\[g(0)=\int \widehat g(\xi) d\xi= \sum_{j=-\infty}^\infty \int \phi(2^j\xi)\widehat g(\xi) d\xi= \sum_{j=-\infty}^\infty  \int 2^{-j} \widehat\phi(2^{-j} x) g(x) dx.\]
Hence we need only to set $\psi=\widehat \phi$.
\end{proof}

\end{lem}
\begin{lem}\label{pv}
There is an odd function $\psi \in \cS (\R)$ of which  Fourier transform  $\widehat{\psi}$ is supported in $[-2,-1/2]\cup[1/2,2]$ such that, for all $g\in \cS(\R)$,
\begin{align}
\label{principle} \operatorname{p.v.}\int\frac{1}{x}\,g(x)dx=\sum_{j=-\infty}^\infty2^{-j}\int\psi(2^{-j}x)g(x)dx .
\end{align}
\end{lem}
\begin{proof}
Let $\chi$ be a smooth function supported in the interval $[1,2]$ satisfying  $\int_\R \chi(x) dx=1/2$. We set $\widehat{\phi}(\xi)=\chi(\xi)+\chi(-\xi)$ and $\varphi(x)=\phi(x/2)-\phi(x)$.
Since $\phi (0)=1$ and $\phi \in \cS (\R)$,  it is easy to see that
\[
\sum_{j=-\infty}^{\infty}\varphi (2^{-j}x) = \lim_{\substack{m\to \infty \\ n\to \infty}} \sum_{j=-m}^{n}\varphi (2^{-j}x) = \lim_{\substack{m\to \infty \\ n\to \infty}} \big(\phi(2^{-n-1}x) -\phi (2^m x)\big) =1
\]
whenever $x\ne 0$. Let us set $\chi_0=\chi_{(-1,1)}$  and $\chi_\infty=1-\chi_0$.  Then,   for $g\in \cS (\R)$,
\[\operatorname{p.v.} \int \frac{1}{x}g(x)dx = \int \frac{1}{x} g(x)\chi_\infty(x) dx+\int \frac{1}{x}(g(x)-g(0)) \chi_0(x)dx.\]
Since $\frac{1}{x} g(x)\chi_\infty(x)+\frac{1}{x}(g(x)-g(0)) \chi_0(x)$ is integrable on $\R$, by the dominated convergence theorem we may write
\[\operatorname{p.v.} \int \frac{1}{x}g(x)dx =\sum_{j=-\infty}^{\infty}  \int \frac{1}{x} \varphi (2^{-j}x)[g(x)\chi_\infty(x)+ (g(x)-g(0)) \chi_0(x)]dx.\]
 Since $\varphi(0)=0$ and $\varphi$ is even, $\frac{1}{x} \varphi (2^{-j}x)$ is integrable  and $\int \frac{1}{x} \varphi (2^{-j}x)\chi_0(x) dx=0$. Thus, we get
 \[\operatorname{p.v.} \int \frac{1}{x}g(x)dx =\sum_{j=-\infty}^{\infty}  \int \frac{1}{x} \varphi (2^{-j}x)g(x) dx.\]
To get the desired \eqref{principle} we need only to set  \[\psi(x)=\frac{\varphi(x)}x.\]

It  now remains to show that $\supp \widehat{\psi} \subset [-2,-1/2]\cup[1/2,2]$. Since $\widehat{\psi}(t)
=\int e^{-2\pi it x}\frac{\varphi(x)}{x}dx$ it is clear that $\frac{d\widehat{\psi}}{dt}(t) = -2\pi i \widehat{\varphi}(t)$. Hence we may write
\begin{equation}\label{psi}
\widehat{\psi}(t)=-2\pi i\int_{-\infty}^t \widehat{\varphi}(s)ds =-2\pi i\int_{-\infty}^t 2\chi(2s) -\chi(s) +2 \chi (-2s) -\chi (-s) ds.
\end{equation}
Since $\chi$ is supported in $[1,2]$, it is easy to check that the integral vanishes if $|t|\ge 2$ or $|t|\le 1/2$. From this it follows that  $\psi\in \cS(\R)$.
\end{proof}

\subsection{ Estimates associated with the surfaces $\Sigma_\rho$}\label{tdel} In sections 3 and 4 we shall apply smooth partition of unity and change of coordinates so that the surface $\{\xi:Q(\xi)=\rho \}$ is written locally as the graph  of
\[\mathcal G_\rho(\tilde{\eta})=\frac{|{\eta'}|^2-|{\eta''}|^2+\rho}{2\eta_1}\]
over the set
\begin{equation*}
{\mathcal{D}}=\big\{\tilde{\eta}=(\eta_1, \eta', \eta'')\in\R^{d-1}: |{\eta'}|\le 1, |{\eta''}|\le 1, \eta_1\in [1,2]\big\}.
\end{equation*}

\begin{lem} \label{decay} Let $\mathcal G_\rho$ be given as in the  above and set
\[
I(x)=\int e^{{2\pi} i(\tilde{x}\cdot\tilde{\eta} +x_d \mathcal G_\rho(\tilde \eta))} \tilde{\chi}(\tilde{\eta}) d \tilde\eta,
\]
{where $\tilde{\chi}\in C^\infty_c(\mathcal{D})$.} Then, there is a constant $C$, independent of $\rho$, such that
\begin{align}
 \label{d-1decay}
|I(x)|&\le C (1+ {|x_d|}{|\rho|})^{-\frac12}(1+|x_d|)^{-\frac{d-2}2}.\end{align}
\end{lem}

\begin{proof} We may assume $\tilde{\chi}(\tilde{\eta})=
{\phi_1(\eta_1)\phi_2 (\eta_1, {\eta'},{\eta''})}$ with $\phi_1\in C^\infty(\mathbb R)$ supported in $[1/2,4]$ and $\phi_2\in  C^\infty(\mathbb R^{d-1})$ supported in $\mathcal D$. Let us write
\[I(x)=\int e^{2\pi i(x_1\eta_1 + x_d \frac \rho{2\eta_1}) } \phi_1(\eta_1) \iint e^{ 2\pi i(x'\cdot \eta' +x''\cdot \eta''
+x_d \frac{|{\eta'}|^2-|{\eta''}|^2}{2\eta_1})}  \phi_2 ( \eta_1, {\eta'},{\eta''}) d {\eta'} d{\eta''} d\eta_1.\]
By Plancherel's theorem the inner integral equals
\[
c \Big(\frac {\eta_1}{|x_d|}\Big)^{\frac{d-2}{2}} \iint e^{-\pi i\frac {\eta_1}{x_d} (|y'|^2-|y''|^2)}
\Phi_{\eta_1}({{y}',{y''}}) d {y'} d{y''},
\]
where $\Phi_{\eta_1}(y',y'')=e^{-\pi i\frac{\eta_1}{x_d}(|x'|^2-|x''|^2+2x'\cdot y'-2x''\cdot y'')}\mathcal F(\phi_2(\eta_1, \cdot, \cdot))(y',y'')$ and $c$ is a constant with $|c|=1$. Hence
\[
I(x)= c |x_d|^{-\frac{d-2}2} \iint \Big( \int\eta_1^\frac{d-2}2
e^{ 2\pi i( x_1\eta_1 +x_d \frac \rho{2\eta_1} -\frac {\eta_1}{2x_d} (|{y'}|^2-|{y''}|^2))}
 \phi_1(\eta_1)\,\,\Phi_{\eta_1}({{y'},{y''}}) d\eta_1 \Big)d {y'} d{y''}  .
\]
 By the van der Corput lemma the inner integral is bounded by $C (1+ {|x_d|}{|\rho|})^{-\frac12}$ (e.g. \cite[Corollary in p.334]{Stein2}). Hence the desired bound follows.
\end{proof}

Let us consider the evolution operator $U_{\rho}(t)$ which is given by
\[
U_\rho(t) g(\tilde x)= \int e^{ {2\pi i}  (\tilde x\cdot \tilde \eta+ t \mathcal G_\rho(\tilde \eta))} \tilde \chi(\tilde\eta)\, \widehat g(\tilde \eta) d\tilde \eta .  \]
From \eqref{d-1decay} we have $\|U_\rho (t) g\|_\infty\lesssim |t|^{-\sigma}|\rho|^{\frac{d-2}2-\sigma} \|g\|_1$ for $\frac{d-2}2\le \sigma \le \frac{d-1}2$. Using  the standard $TT^*$ argument (or following the argument in \cite{KT}) we have, for $\frac{d-2}2\le \sigma \le \frac{d-1}2$,
\begin{equation}\label{tomst}
\|U_\rho(t) g(\tilde x)\|_{L^\frac{2(\sigma+1)}{\sigma}(dt d\tilde x)} \lesssim  |\rho|^{\frac{1}{2(\sigma+1)}(\frac{d-2}2-\sigma)} \|g\|_2.
\end{equation}
In fact, with $\sigma=\frac{d-1}2$, $\sigma=\frac{d-2}2$ we have the estimates
$ \| U_\rho(t) g(\tilde x) \|_{L^\frac{2(d+1)}{d-1}(dt d\tilde x)} \lesssim  |\rho|^{-\frac{1}{2(d+1)}} \|g\|_2,$
$\| U_\rho(t) g(\tilde x) \|_{L^\frac{2d}{d-2}(dt d\tilde x)} \lesssim  \|g\|_2$, respectively.  Interpolation of these estimates also gives \eqref{tomst}.

\

Let $m$ be a smooth function on $\mathcal D$ satisfying
\[ \frac12\le  m \le 2.\]
For $\lambda >0$, we define a multiplier  operator $\T^\rho_\lambda$ by
\begin{equation}\label{mainop}
\wh{\T^\rho_\lambda f}(\eta)=\tilde{\chi}(\tilde{\eta}) {\psi}\big(\lambda^{-1} m(\tilde\eta)(\eta_d -\mathcal G_\rho(\tilde{\eta}))\big)\widehat{f}(\eta),
\end{equation}
where $\psi\in \cS(\mathbb R)$ and $\tilde{\chi}$ is a smooth function supported in $\mathcal{D}$.

\begin{lem}\label{Tomas-Stein}
Let $0<\lambda \le 1$, $\psi\in \cS(\mathbb R)$ and $\T^\rho_\lambda$ be defined by \eqref{mainop}. Then, for $\frac{d-2}2\le \sigma \le \frac{d-1}2$,  the estimate
\begin{equation}\label{TS}
\|\T^\rho_\lambda f\|_{\frac{2(\sigma+1)}{\sigma}} \le C  |\rho|^{\frac{1}{2(\sigma+1)}(\frac{d-2}2-\sigma)}  \lambda^{\frac12} \|f\|_2
\end{equation}
holds with the constant  $C$ independent of $\rho$ and $\lambda$.
\end{lem}
\begin{proof} 
Let $\beta$ be a smooth function supported on $[-2,-1/2]\cup [1/2,2]$ and $\beta_0$ be a smooth function supported on $[-2,2]$ which satisfy $\beta_0(t) + \sum_{j=1}^{\infty}\beta(2^{-j}t)=1$ on $ \R.$ By using this, we  decompose the operator $\T^\rho_\lambda$ so that
\[ \T^\rho_\lambda f = \sum_{j=0}^\infty \T_jf,\]
where $\T_j f$, $j\ge 0$, is defined by  $\wh{\T_0f}(\eta)={\beta_0}\big(\lambda^{-1} (\eta_d -\mathcal G_\rho(\tilde{\eta}))\big)\wh{\T^{\rho}_\lambda f}(\eta)$ and
\begin{align*}
\wh{\T_j f}(\eta)&={\beta}\big((2^j\lambda)^{-1} (\eta_d -\mathcal G_\rho(\tilde{\eta}))\big)\wh{\T^{\rho}_\lambda f}(\eta),\quad  ~ j\ge 1.
\end{align*}
So, it suffices to show that,  for $j\ge 0$,
\begin{equation} \label{l22}
\|\T_j f\|_\frac{2(\sigma+1)}{\sigma}
 \lesssim 2^{-j}|\rho|^{\frac{1}{2(\sigma+1)}(\frac{d-2}2-\sigma)}  \lambda ^{\frac12} \|f\|_2 .
\end{equation}

For  $j\ge 1$, by changing  variables $\eta_d \to\eta_d +\mathcal G_\rho(\tilde{\eta})$, we have
\begin{equation*}
\T_j f(x)=\int e^{2\pi i\eta_d x_d} \beta(\frac{\eta_d}{2^j\lambda})  \int e^{2\pi ix\cdot(\tilde{\eta}, \mathcal G_\rho(\tilde{\eta}))} \tilde{\chi}(\tilde{\eta}) \psi(\frac{m(\tilde\eta) \eta_d}\lambda)\widehat{f}(\tilde{\eta}, \eta_d+\mathcal G_\rho(\tilde{\eta})) d\tilde{\eta} d\eta_d .
\end{equation*}
We observe that the  inner integral  equals
\begin{equation}\label{inner}
U_\rho(x_d) \big( \mathcal F^{-1}_{\tilde x} \big(\psi(\frac{m(\cdot) \eta_d}\lambda){\widehat f}(\cdot, \eta_d+\mathcal G_\rho(\cdot))\big) \big)(\tilde x).
\end{equation}
Here  $\mathcal F^{-1}_{\tilde x} h$ is the inverse Fourier transform of $h$ in $\tilde x$.
By \eqref{tomst} and Plancherel's theorem, we see that the $L_{x}^\frac{2(\sigma+1)}{\sigma}$-norm of \eqref{inner} is bounded by
\[
C|\rho|^{\frac{1}{2(\sigma+1)}(\frac{d-2}2-\sigma)} \|\psi(\frac{m(\cdot) \eta_d}\lambda)\widehat{f}(\cdot, \eta_d+\mathcal G_\rho(\cdot)) \|_{L^2(\R^{d-1})}.\]
Thus, using Minkowski's inequality we get
\[\|\T_j f\|_\frac{2(\sigma+1)}{\sigma}
\lesssim  |\rho|^{\frac{1}{2(\sigma+1)}(\frac{d-2}2-\sigma)} \int |\beta(\frac {\eta_d}{2^j\lambda})|
 \|\psi(\frac{m(\cdot) \eta_d}\lambda)\widehat{f}(\cdot, \eta_d+\mathcal G_\rho(\cdot)) \|_{L^2(\R^{d-1})} d\eta_d.
 \]
Note that $ |\psi(t)| \lesssim |t|^{-2}$ if $|t|\ge 1/2.$ Since $m\sim 1$ and
 $\supp\beta(2^{-j}\cdot)\subset [2^{j-1},2^{j+1}]$,
 $ |\psi(\frac{m(\tilde{\eta})\eta_d}{\lambda}) | \lesssim 2^{-2j},$ whenever $\beta(\frac {\eta_d}{2^j\lambda}) \neq 0.$
Using  this and the Cauchy-Schwarz inequality,  we get, for $j\ge 1$,
\begin{align*}
\|\T_j f\|_\frac{2(\sigma+1)}{\sigma}
&
\lesssim 2^{-2j} |\rho|^{\frac{1}{2(\sigma+1)}(\frac{d-2}2-\sigma)} \int |\beta(\frac {\eta_d}{2^j\lambda})|
 \|\widehat{f}(\cdot, \eta_d+\mathcal G_\rho(\cdot)) \|_{L^2(\R^{d-1})} d\eta_d \\
& \lesssim 2^{-j}|\rho|^{\frac{1}{2(\sigma+1)}(\frac{d-2}2-\sigma)}  \lambda ^{\frac12} \left ( \int \| \widehat{f}(\cdot, \eta_d+\mathcal G_\rho(\cdot))\|^2_{L^2(\R^{d-1})} d\eta_d \right )^\frac{1}{2}.
\end{align*}
 By reversing the change of variables and Plancherel's theorem, the last integral is clearly bounded by $\| f\|_2^2$.  Hence, we get \eqref{l22} for $j\ge 1$.

 Similarly, repeating the same argument one can easily show \eqref{l22}  for $j=0$. So, the proof is completed.
\end{proof}

In the following lemma we obtain an estimate for the kernel of $\T_\lambda^\rho$. For this  the support property of $\widehat \psi$ becomes important in that the estimate \eqref{Kdecay} is no longer true for a general $\psi\in \cS(\R)$.

\begin{lem} \label{kernel}
For every $\rho\neq 0 $ and $0<\lambda \le 1$, let $\mathcal{K}_\lambda^\rho$ be the kernel of $\T_\lambda^\rho$, i.e.,
\begin{equation*}
\mathcal{K}^\rho_\lambda(x)=\int_{\mathbb{R}^d}\psi\big(\lambda^{-1} m(\tilde\eta)(\eta_d-\mathcal G_\rho(\tilde{\eta}))\big) \tilde{\chi}(\tilde{\eta}) e^{2\pi ix\cdot\eta} d\eta ,
\end{equation*}
where $\psi\in \cS(\mathbb R)$ and $\tilde{\chi}$ is a smooth function supported on $\mathcal{D}$.
Suppose $\widehat{\psi}$ is  supported on $\{t: 1/2\le  |t| \le 2\}$.
Then $\mathcal{K}^\rho_\lambda$ is supported in the set $\{x \in \R^d: |x_d| \sim \lambda^{-1}\}$ and
\begin{equation} \label{Kdecay}
|\K^\rho_\lambda(x)| \le C \lambda^{\frac d2} \min(1, \lambda^\frac12|\rho|^{-\frac12}) .
\end{equation}
\end{lem}
\begin{proof} By inversion we write \[\psi\big(\lambda^{-1} m(\tilde\eta)(\eta_d-\mathcal G_\rho(\tilde{\eta}))\big)
=\frac\lambda{m(\tilde\eta)} \int e^{2\pi i \tau (\eta_d-\mathcal G_\rho(\tilde{\eta}))}
\widehat \psi(\frac{\lambda\tau}{m(\tilde \eta)}) d\tau.\]
Inserting this and making the change of variables $\eta_d \to \eta_d + \mathcal G_\rho(\tilde{\eta})$ and taking integration in $\eta_d$, we have
\begin{align*}
\mathcal{K}^\rho_\lambda(x)=\lambda \int   \frac 1{m(\tilde\eta)} \widehat \psi\big(\frac{-\lambda x_d}{ m(\tilde \eta)}\big) e^{{2\pi} i(\tilde{x}\cdot\tilde{\eta} +x_d \mathcal G_\rho(\tilde \eta))} \tilde{\chi}(\tilde{\eta}) d \tilde\eta.
\end{align*}
 Since $\widehat \psi$ is supported in $\{|t|\sim 1\}$ and $m\sim 1$ on the support of $\tilde \chi$, we may assume $|\lambda x_d|\sim 1$ because $\mathcal{K}^\rho_\lambda(x)=0$ otherwise.
Hence we set \[\chi(\tilde \eta)= \frac 1{m(\tilde\eta)} \widehat \psi\big(\frac{-\lambda x_d}{ m(\tilde \eta)}\big) \tilde{\chi}(\tilde{\eta}).\] Then
$\chi(\tilde \eta)$ is  contained in $C_c^\infty(\mathcal D)$ uniformly in  $x_d,\lambda$.  Hence we may repeat the argument in the proof of Lemma \ref{decay} to see that
\[\Big|\int   e^{{2\pi} i(\tilde{x}\cdot\tilde{\eta} +x_d \mathcal G_\rho(\tilde \eta))} {\chi}(\tilde{\eta}) d \tilde\eta\Big|\lesssim (1+ {|x_d|}{|\rho|})^{-\frac12}(1+|x_d|)^{-\frac{d-2}2} . \]
This gives the desired estimate \eqref{Kdecay} because  $|\lambda x_d|\sim 1$.
\end{proof}

\begin{prop}\label{tomst1} Let $\lambda >0$, $0<|\rho|\lesssim 1$, and  $\psi\in \cS(\mathbb R)$ with $\widehat \psi$ supported in $[-2,-1/2]\cup [1/2,2]$ and let $\T^\rho_\lambda$ be defined by \eqref{mainop}. Then, for  $ 1\le p\le 2$ and $\frac{1}{q}=\frac{d-1}{d+1}(1-\frac{1}{p})$, 
\begin{equation}\label{TS1}
\|\T^\rho_\lambda f\|_{q} \lesssim  |\rho|^{-\frac{1}2(\frac1p-\frac1q)}
  \lambda^{\frac{d}p-\frac{d-1}2} \|f\|_p,
 \end{equation}
  and,  for  $ 1\le p\le 2$ and $\frac{1}{q}=\frac{d-2}{d}(1-\frac{1}{p})$,
  \begin{equation}\label{TS2}
\|\T^\rho_\lambda f\|_{q} \lesssim    \lambda^{\frac{d-1}p-\frac{d-2}2} \|f\|_p.
 \end{equation}
\end{prop}

\begin{proof}  We may assume $\lambda\le 1$. Otherwise,  the $L^p$-$L^q$ bound for the multiplier operator is  uniformly bounded because the multiplier is smooth and uniformly bounded in $C^\infty$. The estimate   \eqref{Kdecay} gives
the estimates $\|\T^\rho_\lambda f\|_{\infty} \lesssim  \lambda^{\frac d2} \|f\|_1$ and  $\|\T^\rho_\lambda f\|_{\infty} \lesssim  \lambda^{\frac {d+1}2} |\rho|^{-\frac12}\|f\|_1$. Then interpolation between the first and \eqref{TS} with $\sigma=\frac{d-2}2$ gives  \eqref{TS2}.  Similarly  we interpolate
the second estimate and \eqref{TS} with $\sigma=\frac{d-1}2$ to get \eqref{TS1}.   \end{proof}

\section{Restriction-extension estimate}\label{res-ex}

In this section  we study $L^p$-$L^q$ boundedness of the operator $
f\to \finv{ \delta(Q-\rho) \widehat{f}} $. In fact, we prove Proposition \ref{restriction}  below and  investigate  the allowable range of $p, q$ on which the operator is bounded from $L^p$ into $L^q$.

Recall $\Sigma_\rho=\{\xi: Q(\xi)=\rho\}, \rho\neq0$ and let $d\sigma_\rho$ be the surface  measure  (induced Lebesgue measure) on $\Sigma_\rho$.
To begin with, we note that
\begin{equation}\label{measure}
\int e^{2\pi ix\cdot \xi} \delta(Q(\xi)-\rho) \widehat{f}(\xi) d\xi =\int_{\Sigma_\rho} e^{2\pi ix\cdot \xi} \widehat{f}(\xi) \frac{d\sigma_\rho}{|\nabla Q(\xi)|}.
\end{equation}

\begin{prop}\label{restriction} Let $\rho=\pm 1$.
Under the same assumption as in  Theorem \ref{uniSob1},  we have \eqref{TT*0} whenever
$(1/p,1/q)$ is contained in $\mathfrak T$.  Additionally, if $(1/p,1/q)=B, C, C',$ and $B'$, then
we have $L^{p,1}$-$L^{q,\infty}$ estimate.
\end{prop}

When $Q(D)=\frac{1}{4\pi^2}\Delta$   and $\rho$ is negative real number, then \eqref{TT*00} is  an estimate for the Bochner-Riesz operator of order $-1$.
$L^p$-$L^q$ estimates for the Bochner-Riesz operator of negative order have been studied by several authors {(\cite{Bo, So, CS, bmo})}.  The early results go back as far as  Tomas and Stein (\cite{T, S}).  It was shown  that
\eqref{TT*00} holds for  $p = (2d + 2)/(d+ 3)$ and $q=(2d + 2)/(d-1)$, which is equivalent to $L^{(2d + 2)/(d+ 3)}$-$L^2$ restriction estimates for the sphere.  The estimate \eqref{TT*00} is now known  on the optimal range of $p$ and $q$. That is,
for $1\le p,  q\le \infty$, $Q(D)=\frac{1}{4\pi^2}\Delta$ and $\rho=-1$,  \eqref{TT*00}  is true if and only if  $(1/p,1/q)$ is in the set
\[
\Big\{ \Big(\frac1p,\frac1q\Big)\in[0,1]\times[0,1]: \frac1p-\frac1q \ge\frac{2}{d+1}, \,\,\, \frac1p>\frac{d+1}{2d}, \,\,\, \frac1q<\frac{d-1}{2d} \Big\}.
\]

On the other hand, if $Q(D)$ is not the Laplace operator, the inequality \eqref{TT*0} is known to be true if $p=2d/(d+2)$ and $q=2d/(d-2)$ (the point $F$ in Figure \ref{fig}), which is due to Strichartz \cite{St}.  As is mentioned before, for the special case  $Q(\xi)=-\xi_1^2+ \xi_2^2+\dots+ \xi_2^2$  there are other available estimates  \cite{msw, Na-Oz, Ta}.

\subsection{Proof of Proposition \ref{restriction}}
\label{rest-ext}
We prove Proposition \ref{restriction} by showing the restricted weak type estimates at the endpoints $B$, $B'$, $C$, and $C'$ in Figure \ref{fig}. Then real interpolation between these estimates gives the estimate \eqref{TT*0}  for  $(\frac1p, \frac1q)\in \mathfrak T$.  By  duality it is sufficient to show that, for $(1/p, 1/q)=B,$ $C$,
\begin{equation}\label{decomj}
\|\mathcal F^{-1}\big(\delta(Q\pm 1) \widehat f\,\big)\|_{q,\infty}\le C\|f\|_{p,1}.
\end{equation}

Let us define the  projection operator $P_j$, $j\in\mathbb{Z}$, by \[\widehat{P_jf}(\xi)=\beta(2^{-j}|\xi|)\widehat{f}(\xi),\]
where $\be : \R_+ \to [0,1]$ is a smooth function supported on the interval $[1/2, 2]$ satisfying $\sum_{j=-\infty}^{\infty} \be (2^{-j}t)=1$ for $t>0$. Since $1<p\le 2\le q<\infty$, by Littlewood-Paley theory and Minkowski inequality it is sufficient to show
\begin{equation}\label{decomjj}
\|\mathcal F^{-1}\big(\delta(Q\pm 1) \widehat {P_j f}\,\big)\|_{q,\infty}\le C\|P_jf\|_{p,1}.
\end{equation}%%
To see this we need the following simple lemma.

\begin{lem}\label{little} Let $1<p<\infty$, $1\le r\le \infty$, and let $L^{p,r}$ denote the Lorentz spaces.  Then
\[\|f\|_{p,r}\lesssim \|(\sum |P_j f|^2)^\frac12\|_{p,r}\lesssim \|f\|_{p,r}.\]
\end{lem}

The upper bound follows from the usual Littlewood-Paley inequality $\|(\sum |P_j f|^2)^\frac12\|_{p}\lesssim \|f\|_{p}$, $1<p<\infty$  and (real) interpolation. Once the upper bound is obtained, the lower bound can be shown by using the usual polarization argument.  For example, see \cite{St-singular} or \cite{kl} for detail.  

Hence,  in particular
\[\|\finv{\delta(Q\pm 1)\widehat f\,}\|_{q,\infty} \sim \Big\|\Big(\sum_{j}|\finv{\delta(Q\pm 1)\widehat {P_j f}\,}|^2\Big)^{1/2}\Big\|_{q,\infty}.\]
Since $q>2$, $L^{q/2,\infty}$ is normable. So, we have for $2<q<\infty$
 \begin{equation}\label{mink}\|(\sum_j |h_j|^2)^\frac12\|_{q,\infty}\lesssim (\sum_j \|h_j\|^2_{q,\infty})^\frac12. \end{equation}
Combining this with the above inequality gives
\begin{align*}
\|\finv{\delta(Q\pm 1)\widehat f\,}\|_{q,\infty}
&\lesssim
 \Big(\sum_{j}  \|\mathcal F^{-1}\big(\delta(Q\pm 1)\widehat{P_j f}\,\big)\|_{q,\infty}^2\Big)^{1/2}.
\end{align*}
We now use \eqref{decomjj} to get
\[\|\finv{\delta(Q\pm 1)\widehat f\,}\|_{q,\infty}  \lesssim \Big(\sum_{j} \|P_jf\|_{p,1}^2\Big)^{1/2}.\]
Since $\|g\|_{r,s}=\sup_{\|h\|_{r',s'} \le 1} | \int  g(x) h(x) dx| $ for $1\le r,s\le \infty$,
by the standard duality argument one can easily see that \eqref{mink} implies
$(\sum_j \|h_j\|^2_{p,1})^\frac12 \lesssim \|(\sum_j |h_j|^2)^\frac12\|_{p,1}$ if $1<p<2$. Hence we have
\[\|\finv{\delta(Q\pm 1)\widehat f\,}\|_{q,\infty}  \lesssim \Big\|(\sum_{j} |P_jf|^2)^{1/2}\Big\|_{p,1}.\]
Now Lemma \ref{little} gives \eqref{decomj}. Therefore we are reduced to showing \eqref{decomjj}.

\

Note that we may assume $2^j\ge 2^{-2}$ because $\finv{\delta(Q\pm 1)\widehat{P_j f}\,}=0$, otherwise. Let us set \[\mathbb A=\{\xi\in\mathbb{R}^d:1/2\le|\xi|\le2\}.\] Then by scaling, \eqref{decomjj} is equivalent to
\begin{equation}\label{decom0}
\|\mathcal F^{-1}\big(\delta(Q\pm 2^{-2j}) \widehat f\,\big)\|_{q,\infty}\le C2^{j(2-\frac dp+\frac dq)}\|f\|_{p,1}, \quad \supp \widehat{f}\subset \mathbb A.
\end{equation}

By finite decomposition of $\widehat f$, we may assume that $\widehat f$ is supported in a small neighborhood of {a point} $\xi_0 \in \mathbb A$.
For every invertible linear map $L$ defined on $\mathbb{R}^d$ with $|\det L|=1$, the change of variable $\xi\to L\xi$ in the frequency domain is harmless. Specifically, we apply a rotation $R=R_1\oplus R_2\in SO(\mathbb{R}^d)$, where $R_1\in SO(\mathbb{R}^k)$ and $R_2\in SO(\mathbb{R}^{d-k})$ by splitting  the variable $\xi=(\xi_1,\xi',\xi'',\xi_d)\in\mathbb{R}^k\times\mathbb{R}^{d-k}$ so that the support of $\widehat f$ is contained  in a small neighborhood (in $\mathbb{R}^d$) of  the intersection of
$
\{ \xi : Q(\xi)=\mp 2^{-2j},~ 1/2 \le |\xi|\le2, ~{\xi_1\ge 0, ~\xi_d\ge 0} \}
$
and the $\xi_1\xi_d$-plane.
Since $Q(R\xi)=Q(\xi)$, the surface $\{\xi :Q(\xi)=\mp 2^{-2j} \}$ and the measure $\delta(Q(\xi)\pm 2^{-2j}) d\xi=
\frac{d\sigma_{\mp2^{-2j}}}{|\nabla Q(\xi)|}$ are invariant under the rotation $R$.  We may assume that  the surface is given by
\begin{equation}\label{change}
\mp\,2^{-2j}=Q(\xi)=(\xi_d+\xi_1)(\xi_d-\xi_1)-\xi_2^2-\cdots -\xi_{k}^2+\xi_{k+1}^2+\cdots+\xi_{d-1}^2,
\end{equation}
and that $\widehat f$ is supported on the set \[\{ \xi \in \R^d:   |\xi_d+\xi_1|\sim 1,\, |\xi_d-\xi_1| \lesssim 1,\, |\xi'|\ll 1, \, |\xi''|\ll 1\}.  \]

Next we apply another harmless change of variables via the rotation $\xi\to \eta$, where
\begin{equation}\label{chvar}
    \def\arraystretch{1.5}
    \left\{
    \begin{array}{l}
    \eta_1=(\xi_d+\xi_1)/\sqrt{2}, \quad \eta_d=(\xi_d-\xi_1)/\sqrt{2}, \\
    {\eta'}%=(\eta_2,\cdots,\eta_k)=(\xi_2,\cdots ,\xi_k)
    ={\xi'},\
    {\eta''}%=(\eta_{k+1},\cdots ,\eta_{d-1})=(\xi_{k+1},\cdots ,\xi_{d-1})
    ={\xi''}.
    \end{array}
    \right.
\end{equation}
As mentioned in the introduction, for notational convenience we  write
$\tilde{\eta}=(\eta_1,{\eta'},{\eta''})=(\eta_1,\cdots,\eta_{d-1})\in\mathbb{R}^{d-1}$.
Then the surface given by \eqref{change} is now represented locally as the graph of $\mathcal G_{ \mp 2^{-2j}}$ in the new coordinate.
For the rest of this section we set
\[\rho=\mp 2^{-2j}.\]

Hence, by change of variables, \eqref{decom0} is again equivalent to the estimate
\begin{equation}\label{TT*3}
\Big\| \int  \delta({\eta_d-\mathcal G_\rho(\tilde{\eta})}) e^{2\pi ix\cdot \eta} \widehat{f}(\eta)\chi(\eta)d\eta\Big\|_{L^{q,\infty}(\mathbb{R}^d)} \le C |\rho|^{\frac{1}{2}(\frac{d}{p}-\frac{d}{q}-2)}\|f\|_{L^{p,1}(\mathbb{R}^d)},
\end{equation}
where $\chi$ is a smooth function supported in a set
$
\mathbb{A}_0:= \mathbb A\cap\{\eta\in\mathbb{R}^d : |{\eta'}|\ll1,~ |{\eta''}|\ll 1,  |\eta_d| \lesssim 1,~ \eta_1\sim 1 \}.
$
We now use Lemma \ref{del} to get
\begin{align*}
 \finv{ \delta({\eta_d-\mathcal G_\rho(\tilde{\eta})}) \widehat{f}(\eta)\chi(\eta)}
= \sum_{l\in\mathbb{Z}} T_l f,
\end{align*}
where $\psi$ is a smooth function with $\supp\widehat{\psi}\subset\{t\in\mathbb{R}:|t|\sim 1\}$ and
\[T_l f(x)= 2^{-l}\int_{\mathbb{R}^d}
       \psi(2^{-l}(\eta_d-\mathcal G_\rho(\tilde{\eta})))\chi(\eta)\widehat{f}(\eta)e^{2\pi ix\cdot\eta}d\eta . \]

\subsubsection{Restricted weak type estimate at $B$}\label{rest-b}
By Proposition \ref{tomst1} with $m=1$ we have, for $ 1\le p\le 2$ and $\frac{1}{q}=\frac{d-2}{d}(1-\frac{1}{p})$,
\begin{equation}\label{TS22}
\|{T}_l f\|_{q} \lesssim    2^{l(\frac{d-1}p-\frac{d}2)} \|f\|_p.
\end{equation}

Now we make use of  the following elementary lemma which was implicit in \cite{B}. A statement in more general setting can also be found in \cite{CSWaWr}.
\begin{lem}\label{interpol}
Let $\e_0, \e_1 > 0$, and let $\{T_l:l\in\Z\}$ be a sequence of linear operators satisfying
\[
\|T_l f\|_{q_0} \le M_0 2^{\e_0{l}}\|f\|_{p_0}, \quad \|T_l f\|_{q_1} \le M_ 1 2^{-\e_1{l}}\|f\|_{p_1}
\]
for some $1\le p_0, p_1, q_0, q_1\le \infty$. Then $T=\sum_{l\in \Z}T_l$ is bounded from $L^{p,1}$ to $L^{q,\infty}$ with
$
\|T f\|_{q,\infty} \le C M_0^\theta M_1^{1-\theta} \|f\|_{p,1},
$
where $\theta =\e_1/(\e_0+\e_1)$, $1/q=\theta/q_0+(1-\theta)/q_1$, $1/p = \theta / p_0 +(1-\theta) / p_1$.
\end{lem}

We choose $(p_i, q_i)$ satisfying $\frac{1}{q_i} =\frac{d-2}{d}(1-\frac{1}{p_i})$ for $i=0,1$, and $\frac{1}{2}\le \frac{1}{p_0} < \frac{d}{2(d-1)} <\frac{1}{p_1} <1$. Then by \eqref{TS22} we have for $i=0,1,$
\[
\|T_l f\|_{q_i} \le C 2^{ l(\frac{d-1}{p_i}-\frac{d}{2})} \| f\|_{p_i}.
\]
Note that $\frac{d}{2}-\frac{d-1}{p_1}<0<\frac{d}{2}-\frac{d-1}{p_0}$. We combine two estimates for $i=0,1$  and  Lemma \ref{interpol} with $\e_0 = \frac{d}{2}-\frac{d-1}{p_0}$ and $\e_1=-\frac{d}{2}+\frac{d-1}{p_1}$ to get
\[\|\finv{ \delta({\eta_d-\mathcal G_\rho(\tilde{\eta})}) \widehat{f}(\eta)\chi(\eta)}\|_{q,\infty}\le C\|f\|_{p,1}\]
for $(1/p, 1/q)=B$. This is \eqref{TT*3} when $(1/p, 1/q)=B$.

\subsubsection{Restricted weak type estimate at  $C$}\label{rest-c}
Again by proposition \ref{tomst1} we have, for  $ 1\le p\le 2$ and $\frac{1}{q}=\frac{d-1}{d+1}(1-\frac{1}{p})$,
\begin{equation*}
\|{T}_l f\|_{q} \lesssim  |\rho|^{-\frac{1}2(\frac1p-\frac1q)}
  2^{l(\frac{d}p-\frac{d+1}2)} \|f\|_p.
 \end{equation*}
By the same argument as before we get the restricted weak type estimate
\[\|\finv{ \delta({\eta_d-\mathcal G_\rho(\tilde{\eta})}) \widehat{f}(\eta)\chi(\eta)}\|_{q,\infty}\lesssim |\rho|^{-\frac{1}2(\frac1p-\frac1q)} \|f\|_{p,1}\]
for $(1/p, 1/q)=C$. Finally note that $|\rho|^{-\frac{1}2(\frac1p-\frac1q)}= |\rho|^{\frac{1}{2}(\frac{d}{p}-\frac{d}{q}-2)}$ when $(1/p, 1/q)=C$. Hence we have \eqref{TT*3}  for $(1/p, 1/q)=C$.

\subsection{Estimate for $f\to \finv{\psi(2^{-l}(Q-a))\widehat f\,\,}$}  In this section we prove a few estimates which will be used later.

\begin{prop} \label{global-bound}
\label{global} Let $\lambda>0$, $ 0< |a|\lesssim 1$,  $\psi\in \cS(\mathbb R)$ with $\widehat \psi$ supported in $[-2,-1/2]\cup [1/2,2]$.  Then, if the support of Fourier transform of $f$ is contained in $\{ \xi: |\xi|\ge 1/2\}$,  for  $ 1< p\le 2$ and $\frac{1}{q}=\frac{d-1}{d+1}(1-\frac{1}{p})$, 
\begin{equation}\label{psi1}
\| \finv{\psi(\lambda^{-1}(Q-a))\widehat f\,\,}\|_{q} \lesssim  |a|^{-\frac{1}2(\frac1p-\frac1q)}
  \lambda^{\frac{d}p-\frac{d-1}2} \|f\|_p,
 \end{equation}
  and,  for  $ 1<p\le 2$ and $\frac{1}{q}=\frac{d-2}{d}(1-\frac{1}{p})$,
  \begin{equation}\label{psi2}
\|\finv{\psi(\lambda^{-1}(Q-a))\widehat f\,\,}\|_{q} \lesssim    \lambda^{\frac{d-1}p-\frac{d-2}2} \|f\|_p.
 \end{equation}
\end{prop}

In order to show this, by Littlewood-Paley inequality and using the
fact that $1<p\le 2\le q <\infty$, it is sufficient to  obtain
\eqref{psi1}  and \eqref{psi2} for the same $p,q$ as in Proposition
\ref{global}  with  $f$ of which Fourier transform is supported $\{
\xi: 2^{j-1}\le |\xi|\le 2^{j+1}\}$, $j\ge -1$. The estimates for
each dyadic piece can be put together by the same argument as
before.
 By rescaling it is enough to do this with $f$ whose Fourier transform is supported in $\mathbb A$.
In fact, by rescaling ($\xi\to 2^j\xi$ in frequency domain) we have
\[ \finv{\psi(\lambda^{-1}(Q-a))\widehat f\,\,} (x)=\finv{\psi((2^{-2j}\lambda)^{-1}(Q-2^{-2j}a))\widehat {f(2^{-j}\cdot)}}(2^{j} x).\]
Since $\widehat {f(2^{-j}\cdot)}$ is supported in $\mathbb A$, we see that  \eqref{psi1}  and \eqref{psi2}  with $\widehat f$ supported in $\mathbb A$ implies
\begin{align*}
\| \finv{\psi(\lambda^{-1}(Q-a))\widehat f\,\,}\|_{q} &\lesssim 2^{j((d-1)(1-\frac1p)-\frac{d+1}q)}  |a|^{-\frac{1}2(\frac1p-\frac1q)}
  \lambda^{\frac{d}p-\frac{d-1}2} \|f\|_p,\\
  \|\finv{\psi(\lambda^{-1}(Q-a))\widehat f\,\,}\|_{q} &\lesssim   2^{j((d-2)(1-\frac1p)-\frac dq)} \lambda^{\frac{d-1}p-\frac{d-2}2} \|f\|_p,
\end{align*}
respectively,
provided that the Fourier transform of $f$ is supported in $\{ \xi: 2^{j-1}\le |\xi|\le 2^{j+1}\}$, $j\ge -1$. Therefore, for the proof of
Proposition \ref{global} it is sufficient to show \eqref{psi1} with $\widehat f$ supported in $\mathbb A$ and  $0<|a|\lesssim 1$ for  $ 1< p\le 2$ and $\frac{1}{q}=\frac{d-1}{d+1}(1-\frac{1}{p})$,  and \eqref{psi2}
for  $ 1<p\le 2$ and $\frac{1}{q}=\frac{d-2}{d}(1-\frac{1}{p})$.

Since $Q$ is non-elliptic, by finite decomposition of the support of $\widehat f$,  rotation and changing variables (\eqref{change}, \eqref{chvar}), to show \eqref{psi1} and \eqref{psi2}  with $\widehat f$ supported in $\mathbb A$,  it is sufficient to show the same bounds for   $\finv{\psi(\lambda^{-1}(2\eta_1(\eta_d-\mathcal G_a(\tilde\eta)))\tilde \chi \,\widehat f\,\,}$ instead of $\finv{\psi(\lambda^{-1}(Q-a))\widehat f\,\,}$ while $\widehat f$ is assumed to be supported in $\mathbb A _0$.  This can easily  be done
by repeating the proof of Proposition \ref{tomst1} by using Lemma \ref{Tomas-Stein} and Lemma \ref{kernel} with $m(\tilde\eta)=2\eta_1$.  This completes the proof.

\subsection{Bounds for the multiplier given by principal value} \label{pvpv} Let us consider  the estimate
\begin{equation}\label{pva}
\Big\|\F^{-1}\Big(  \operatorname{p.v.}\frac{1}{Q(\xi)\pm 1}  \widehat{f}(\xi) \Big) \Big\|_q \leq C\norm{f}_p.
\end{equation}
We now have another result similar to Proposition \ref{restriction}.
\begin{prop}\label{unipv} Let $d\ge3$ and let $Q$ be a non-elliptic  quadratic form as in Theorem \ref{mainthm}.  {Let $(1/p,1/q)$ be contained in $\mathfrak{T}$.}
Then there is a constant $C$ such that \eqref{pva} holds.  Additionally, if $(1/p,1/q)=B, C, C',$ and $B'$, then
we have $L^{p,1}$-$L^{q,\infty}$ estimate.
\end{prop}

This can be proved by the same argument which is used for the proof of Proposition \ref{restriction}.
So, we shall be brief.    The distribution $\operatorname{p.v.}\frac{1}{Q(\xi)\pm 1} $ is smooth on $|\xi |<3/4$ and bounded away from zero. So, we may assume $\widehat f$ is supported in $\{\xi: |\xi|\ge 1/2\}$.  As before, by Littlewood-Paley
theory and scaling  it is enough to show that, for $(1/p,1/q)=B, C,$ and $j\ge 0$
\begin{equation*}
\Big\|\F^{-1} \Big( \operatorname{p.v.}\frac{1}{Q(\xi)\pm 2^{-2j}} \widehat{f}(\xi)\Big) \Big\|_{q,\infty} \le C2^{j(2-\frac dp+\frac dq)}\|f\|_{p,1},
\quad \supp\widehat{f}\subset \mathbb A.
\end{equation*}

Let us set $\rho=\mp 2^{-2j}$ as before.  By finite decomposition, rotation, and change of variables \eqref{change} and \eqref{chvar}, this further reduces to showing the estimate
\begin{equation}\label{m1_2}
\Big\|\F^{-1} \Big( \operatorname{p.v.}\frac{1}{\eta_d-\mathcal G_\rho(\tilde{\eta})} \tilde \chi(\tilde \eta)\widehat{f}(\eta)\Big) \Big\|_{q,\infty}\le C{|\rho|^{\frac{1}{2}(\frac{d}{p}-\frac{d}{q}-2)} }\|f\|_{p,1}, \quad \supp\widehat{f}\subset \mathbb A_0,
\end{equation}
where  $\tilde \chi$ is smooth function supported in $\mathcal D$. Now, by Lemma \ref{pv}, we decompose this operator as
\[\F^{-1} \Big( \operatorname{p.v.}\frac{1}{\eta_d-\mathcal G_\rho(\tilde{\eta})} \tilde\chi(\tilde\eta)\widehat{f}(\eta)\Big)=\sum_{l\in\mathbb{Z}} 2^{-l}\int_{\mathbb{R}^d}\psi(2^{-l}(\eta_d-\mathcal G_\rho(\tilde{\eta}))) \tilde\chi(\tilde\eta) \widehat{f}(\eta) e^{2\pi ix\cdot\eta} d\eta,
\]
where $\psi$ is a smooth function on $\mathbb{R}$ such that $\supp\widehat{\psi}\subset\{t\in\mathbb{R}:|t|\sim 1\}$. At this point we remark that the exactly same argument as in Section \ref{rest-ext} can be applied to show \eqref{m1_2} for
$(1/p, 1/q)=B$, or $C$.  So we avoid duplication.

\subsection{Necessary conditions }\label{nec-sec}
In this section we obtain necessary conditions for the estimates \eqref{Sob},
\eqref{resol}, \eqref{TT*0}. By the implication  \eqref{Sob} $\rightarrow$ \eqref{resol} $\rightarrow$ \eqref{TT*0} it is sufficient to consider
\eqref{TT*0}.

\subsubsection{Failure of \eqref{TT*0} for $\frac1p-\frac1q>\frac2d$}
After change of variables \eqref{change}, \eqref{chvar}, the quadratic form $Q$ is replaced by $2\eta_1\eta_d -|\eta'|^2+|\eta''|^2$.  By  \eqref{measure}  it follows that the estimate \eqref{TT*0} implies
\begin{equation}\label{modi}
\Big\| \int e^{2\pi i(\tilde{x}\cdot\tilde{\eta}+x_d
\mathcal G_{\pm 1}(\tilde{\eta}))} \widehat{g}(\tilde{\eta},\mathcal G_{\pm 1}(\tilde{\eta})) \frac{1}{2\eta_1}d\tilde{\eta}\Big\|_{L^q(\R^d)}\le C\|g\|_{L^p(\R^d)},
\end{equation}
whenever $\widehat{g}$ vanishes near $\eta_1=0.$
Let $\phi\in C_c^\infty(\R)$ be supported on the interval $[-2^{-2},2^{-2}]$. For $0<\la\ll1$, define $g_\la \in \mathcal{S}(\R^d)$ by
\begin{equation}\label{gla}
\widehat{g_\la}(\eta)=\phi(\la^2(\eta_1-\la^{-2})) \phi(\eta_d)\prod_{j=2}^{d-1}\phi(\la\eta_j).
\end{equation}
Since $\widehat{g_\la}$ is supported in the $d$-dimensional rectangle
\[
R_\la=\{\eta\in\R^d: |\eta_1-\la^{-2}|\le (2\la)^{-2},~|\eta_d|\le 2^{-2},~|\eta_j|\le ({4\la})^{-1},~2\le j\le d-1\},
\]
it is easy to see  that
$ \big| \int  e^{2\pi i(\tilde{x}\cdot\tilde{\eta}+x_d\mathcal G_{\pm 1}(\tilde{\eta}))} \widehat{g_\la}(\tilde{\eta},\mathcal G_{\pm 1}(\tilde{\eta}))\frac{1}{2\eta_1} d\tilde{\eta} \big| \gtrsim \la^2 |{R_\la}|$
if $x$ is in the set
\[
R_\la'=\big\{x\in\R^d : |x_1|\le  \frac{ \la^2}{125d} ,~ |x_d|\le \frac{1}{20d} ,~ |x_j|\le\frac{\la}{25d},~ 2\le j\le d-1\big\}.
\]
Hence, we have
\[
\Big\| \int e^{2\pi i(\tilde{x}\cdot\tilde{\eta}+x_d\mathcal G_{\pm 1}(\tilde{\eta}))}
\widehat{g_\la}(\tilde{\eta},\mathcal G_{\pm 1}(\tilde{\eta}))\frac{1}{2\eta_1}d\tilde{\eta} \Big\|_{L^q(R_\la')}
\gtrsim \la^2 |{R_\la}| |R_\la'|^{1/q}\sim \la^{2-d+d/q}.
\]
On the other hand it is also clear that
$
\|g_\la\|_p\lesssim\la^{-d+d/p}$. Therefore, \eqref{modi} gives $ \la^{2-d+d/q}\lesssim \la^{-d+d/p}$. By letting {$\la\to 0$} we see that  the inequality \eqref{TT*0} cannot be true unless $1/p-1/q\le2/d$.

\subsubsection{Failure of \eqref{TT*0} for $\frac1p-\frac1q< \frac2{d+1}$, $q {~ \le~ }\frac{2d}{d-1} $, $p{ ~\ge ~} \frac{2d} {d+1}$ }
Failure of \eqref{TT*0} on this range can be shown similarly as in the proof of Bochner-Riesz means of negative order  (see \cite{Bo}  \cite{CKLS}).  Here we only consider the case  $\delta(Q-1)$. The other case $\delta(Q+1)$ can be shown via a little modification.

Typical Knapp's example shows
the estimate \eqref{TT*0} is only possible for
\begin{equation}\label{necc} \frac1p-\frac1q\ge \frac2{d+1}.  \end{equation}
In fact, for $0<\la \ll 1$,
{let us define $\widehat{f_\la}(\xi)=\phi(\la^{-2}(\xi_d-1))\prod_{j=1}^{d-1}\phi(\la^{-1}\xi_j)$,
where $\phi\in C^\infty_c(\R)$ is the same function as in \eqref{gla}.
Then it is easy to see
that
\[\Big|\int \delta(Q(\xi)-1) e^{2\pi ix\cdot \xi} \widehat{f_\la}(\xi) d\xi \Big| \gtrsim  \la^{d-1}\]
for $|x_d|\le c\la^{-2}$ and $|x_j|\le c\la^{-1},\, j\neq d$ with a sufficiently small $c>0$.
The estimate \eqref{TT*0} implies $\la^{d-1}\la^{-\frac{d+1}q}\lesssim \la^{d+1-\frac{d+1}{p}} $.}
Letting $\la\to 0$ gives the condition \eqref{necc}.

The surface $\{\xi: Q(\xi)=  1,\  |\xi|\le 2\}$ has nonvanishing
Gaussian curvature.  If we choose a function $f$ with $\widehat f$
supported in a small enough neighborhood of $(0,\pm 1)\in\mathbb
R^{d-1}\times \mathbb R$, then by the stationary phase method  we
see
\[\Big|\int \delta(Q(\xi) \, - \,1) e^{2\pi ix\cdot \xi} \widehat{f}(\xi) d\xi \Big|\gtrsim |{x_d}|^\frac{1-d}2\]
if $c|x_d|>|\tilde x|$. The estimate \eqref{TT*0} implies that $|{x_d}|^\frac{1-d}2\chi_{\{c|x_d|>|\tilde x|\}}\in L^q$.  Hence, it follows that
$\frac{(d-1)q}2>d$ and the estimate \eqref{TT*0} can not be true for $q {~ \le~ }\frac{2d}{d-1} $.
Duality gives  the other condition $p{ ~\ge ~} \frac{2d} {d+1}$.

\subsubsection{Necessity of the  condition $p<2(d-1)/d, ~ q>2(d-1)/(d-2)$ for \eqref{TT*0} when $1/p-1/q=2/d$} \label{nec-scaling}
It is enough to show
$q > 2(d-1)/(d-2)$ because of duality.

Let $m_1< m_2$ be positive numbers. From scaling,  we see that \eqref{TT*0} implies for  any $r>0$
\begin{equation*}
\Big\| \int\delta( Q(\xi)-r^2) e^{2\pi ix\cdot \xi} \wh{f}(\xi)d\xi\Big\|_{L^q(\mathbb{R}^d)}
\le C \|f\|_{L^p(\mathbb{R}^d)}
\end{equation*}
with $C$ independent of $m_1, m_2$ if  $\supp \wh{f}\subset \{\xi: m_1\le |\xi|\le m_2 \}$ and $1/p-1/q=2/d$.  Letting $r\to 0$ gives, for   $\supp \wh{f}\subset \{\xi:  m_1\le |\xi|\le m_2 \}$,
\begin{equation}\label{qzero}
\Big\| \int \delta(Q(\xi)) e^{2\pi ix\cdot \xi} \wh{f} (\xi)d\xi\Big\|_{L^q(\mathbb{R}^d)}
\le C \|f\|_{L^p(\mathbb{R}^d)}.
\end{equation}

  Then parameterizing the set $\{\xi: Q(\xi)=0\}$, in particular,  we see that
\eqref{qzero} implies
\begin{equation}\label{eq:cone}
\Big\| \int_{\R^{d-1}} \wh{g} \Big( \tilde{\eta}, \frac{|{\eta'}|^2-|{\eta''}|^2}{2\eta_1}\Big ) e^{2\pi i (\tilde{x}\cdot\tilde{\eta}+x_d\frac{|{\eta'}|^2-|{\eta''}|^2}{2\eta_1})} \tilde{\chi}(\tilde{\eta})  d\tilde{\eta} \Big\|_q \le C \|g\|_p
\end{equation}
whenever $\tilde \chi$ is a smooth function  supported in $\{\tilde \eta: \eta_1\in (10^{-2}, 10^2),  |(\eta', \eta'')|\le M \} $ for any $M>0$.  Here we use the same coordinates  given  by \eqref{chvar}.
Let $\phi \in C^{\infty}_c(-10^{-2},10^{-2})$ be satisfying
\[ 0\le \wh{\phi}(t) \le 2 ~\text{for ~ all }~t \in\R \quad\text{and} \quad \wh{\phi}(t) \ge 1~ \text{on}~ |t|\le 1 .\]
Also, let $\phi_2 \in C^{\infty}_c(\R^{k-1})$ and $\phi_3 \in C^{\infty}_c(\R^{d-k-1})$ be radial functions which are supported in the balls $B(0,M/2)$ and have nonnegative Fourier transforms.
We set
\[\phi_1(t) = t^{\frac{-d+2}{2}}\phi(t-1), \  \tilde{\chi}(\tilde{\eta})=\phi_1(\eta_1)\phi_2({\eta'})\phi_3({\eta''}).\]
From the support condition, the inequality \eqref{eq:cone} holds for such $\tilde{\chi}$.
Choose $g\in \mathcal{S}(\R^d)$ such that  $|\wh{g} \Big( \tilde{\eta}, \frac{|{\eta'}|^2-|{\eta''}|^2}{2\eta_1}\Big )|\equiv 1$ if $\tilde\eta\in \supp \tilde \chi $. Then \eqref{eq:cone} implies that
\[K(x) = \int_{\R^{d-1}} e^{2\pi i (\tilde{x}\cdot\tilde{\eta}+x_d\frac{|{\eta'}|^2-|{\eta''}|^2}{2\eta_1})} \phi_1(\eta_1)\phi_2({\eta'})\phi_3({\eta''}) d\tilde{\eta}\]
is in  $L^q(\R^d)$.

We now compute $K$.  By making use of the Fourier transform of the Gaussian function, we obtain,
for $x_d\neq 0$,
\begin{align*}
K(x)&=|x_d|^{-\frac{d-2}{2}}e^{\pi i \frac{-d+2k}{4}} \int\int e^{2\pi i \eta_1(x_1 -\frac{|y|^2-|z|^2}{2x_d})}
\eta_1^{\frac{d-2}{2}}\phi_1(\eta_1) d\eta_1 \wh{\phi_2}({y}-{x'})\wh{\phi_3}({z}-{x''}) dydz
\\
&= |x_d|^{-\frac{d-2}{2}}e^{2 \pi i( \frac{-d+2k}{8} +x_1-\frac{|{x'}|^2-|{x''}|^2}{2x_d} )}  \times I(x),
\end{align*}
where
\[I(x)=\int   \wh{\phi_2}(y)\wh{\phi_3}(z)\wh{\phi}(-x_1 +\frac{|y+{x'}|^2-|z+{x''}|^2}{2x_d})
 e^{-2\pi i ( \frac{2{x'}\cdot y-2{x''}\cdot z+|y|^2-|z|^2}{2x_d})}dydz.
\]

Let us set $B=\iint \wh{\phi_2}(y)\wh{\phi_3}(z) dydz$. Then let $\lambda$ be a number large enough such that $\iint_{|(y,z)|\ge \lambda}  \wh{\phi_2}(y)\wh{\phi_3}(z) dydz\le 10^{-2}B.$
Note that, if $|(y,z)|\le \lambda$,
$ |x_1-\frac{|x'|^2-|x''|^2}{2 x_d}|\le 1/2$, $|x'|, |x''|\le  10^{-3}\lambda^{-2}|x_d|$ and { $|x_d| \ge 10^{3} \lambda^2 $},
then
$|-x_1 +\frac{|y+{x'}|^2-|z+{x''}|^2}{2x_d}|\le 1$  and $
|  \frac{2{x'}\cdot y-2{x''}\cdot z+|y|^2-|z|^2}{2x_d}|\le 10^{-2}.$
Hence,  by the choice of $\lambda$
\[ \Big| I(x)-\int_{|(y,z)|\le \lambda} \wh{\phi_2}(y)\wh{\phi_3}(z)\wh{\phi}(-x_1 +\frac{|y+{x'}|^2-|z+{x''}|^2}{2x_d})
dydz\Big |
\le 10^{-1}B\]
provided that $x$ is in the set
\begin{equation*}
U_\lambda=\Big\{x: |x_1-\frac{|x'|^2-|x''|^2}{2 x_d}|\le 1/2, \  |x'|, |x''|\le  10^{-3}\lambda^{-2}x_d,\ {x_d\ge 10^{3}\lambda^{2}}\Big\}.
\end{equation*}
Hence, if $x\in U_\lambda$,
$|I(x)| \ge \frac12 B.$
Therefore, we see that if $x\in U_\lambda$
\[
 |K(x)|\gtrsim B |x_d|^{-\frac{d-2}2}.\] Using this
\begin{align*}
 \int   |K(x)|^q &\gtrsim \int  (\int \int  \chi_{U_\lambda}(x) |x_d|^{-\frac{q(d-2)}2}  dx' dx'' dx_1) dx_d \\
 & \sim  \int_{10^3\lambda^2}^\infty  |x_d|^{-\frac{q(d-2)}2} x_d^{d-1} x_d^{-1} dx_d.  \end{align*}
The last integral must be finite since $K\in L^q$. Hence we get $q > 2(d-1)/(d-2)$ as desired.

\section{Proofs of Theorem \ref{uniSob1} and Theorem \ref{mainthm}}

\subsection{Proof of Theorem  \ref{mainthm}}  The necessity part follows from the scaling condition \eqref{scaling} and  the condition in the subsection \ref{nec-scaling}.  For the proof of the sufficiency part of Theorem \ref{mainthm} by duality and interpolation it is sufficient to show the restricted weak type  bound
\eqref{wk} for  $(p,q)=
\big(\frac{2(d-1)}{d},\frac{2d(d-1)}{(d-2)^2}\big) $. 
By scaling, limiting argument and Lorentz transformation,  to show \eqref{wk} it is enough to show the following (see \cite[Proposition 2.1]{KRS}):

If $\alpha\in\mathbb{R}\setminus\{0\}$, $\beta\in\mathbb{R}$, $\lambda=\pm1$ and $l=d$ (or $l=1$), there exists a uniform constant $C$ such that, for  $(p,q)= \big(\frac{2(d-1)}{d},\frac{2d(d-1)}{(d-2)^2}\big) $,
\begin{equation}\label{reduction1}
\|u\|_{q,\infty}\le C\Big\|\Big(Q(D)+\alpha\Big(\frac{\partial}{\partial x_l}+i\beta\Big)+\lambda \Big)u \Big\|_{p,1}, \quad u\in\mathcal{S}(\mathbb{R}^d).
\end{equation}
Even though  $L^{p,1},$ $L^{q,\infty}$ are used here instead of $L^p,$ $L^q$, the reduction can be justified without modification
by following the argument \cite{KRS}   because $L^{p,1}$ and $L^{q,\infty}$ are normable.

Further reduction is possible by following the argument in  \cite[pp.335--337]{KRS}. In fact,  we make use of Littlewood-Paley inequality (projections in $x_d$) in Lorentz spaces (cf. Lemma \ref{little}) and Proposition \ref{restriction} which gives the restricted weak type estimate for $(p,q)=
\big(\frac{2(d-1)}{d},\frac{2d(d-1)}{(d-2)^2}\big) $. One may repeat the same argument by replacing $L^p$, $L^q$ with $L^{p,1}$, $L^{q,\infty}$. This reduction works well because of the scaling condition $1/p-1/q=2/d$. So,   in order to prove the estimates \eqref{reduction1} it is enough to show that, for $(p,q)=
\big(\frac{2(d-1)}{d},\frac{2d(d-1)}{(d-2)^2}\big) $,
\[\|u\|_{L^{q,\infty}(\mathbb{R}^d)}\le C\|(Q(D)+z)u\|_{L^{p, 1}(\mathbb{R}^d)},\quad z\in\mathbb{C}. \]
Thanks to the scaling condition  \eqref{scaling},  by scaling we only need to show the above estimate for $|z|\ge 1$.    Therefore, it is enough to show Theorem \ref{uniSob1}.  This is done in the  following section.

\medskip

\subsection{Proof of Theorem \ref{uniSob1}} As before, by duality and interpolation it is enough to show, for $(1/p,1/q)=B,$ $C$,
\[
\|u\|_{L^{q,\infty}(\mathbb{R}^d)}\le C\|(Q(D)+z)u\|_{L^{p,1}(\mathbb{R}^d)},\quad |z|\ge1.
\]
Writing $z=a+ib$, this reduces to showing
\begin{equation}\label{local0}
\Big\|\F^{-1}\big(({Q+a+ib})^{-1} {\widehat{f}}\,\,\big) \Big\|_{q,\infty} \leq C\norm{f}_{p, 1}
\end{equation}
which is uniform in $a, b,$ provided $ a^2+b^2=1$. In fact, by scaling \eqref{local0} implies
\[\Big\| \F^{-1}\big(({Q+z})^{-1} {\widehat{f}}\,\, \big) \Big\|_{q,\infty} \leq C |z|^{\frac12(\frac dp-\frac dq -2)}\norm{f}_{p,1}, \quad z\in \C .\]
Since $\frac dp-\frac dq -2\le 0$,  the desired estimate   follows for $(1/p,1/q)=B,$ $C$.  Hence this proves Theorem \ref{uniSob1}.
For the rest of this section we fix $p, q$  so that
\[ \Big(\frac1p,\frac1q\Big)=B,~ C.\]

 Let $\psi$ be a smooth function which is supported in $\{\xi: |\xi|\le 3/4\}$. Since $|a+ib|=1$, $({Q+a+ib})^{-1} \psi$ is a smooth function uniformly contained in $C^\infty$. Thus the multiplier operator $f\to \finv{({Q+a+ib})^{-1} \psi {\widehat{f}\,\,}}$  is uniformly bounded from $L^p$ to $L^q$ for $1\le p\le q\le \infty$. So, for the proof of  \eqref{local0} we may assume
\begin{equation}
\label{supp-con}
\supp \widehat f\subset \big\{\xi: |\xi|\ge \frac12\big\}.
\end{equation}  We separately consider the real and imaginary parts of the multiplier.   Let us write
\begin{equation}\label{re+im}
\frac{1}{Q(\xi)+a+ib} = \frac{Q(\xi)+a}{(Q(\xi)+a)^2+b^2} - \frac{ib}{(Q(\xi)+a)^2+b^2}.
\end{equation}
We also need to use the generalized polar coordinate which is given by the quadratic form $Q$.    Let us set $\Sigma_{\pm}=\{\xi : Q(\xi)=\pm 1\}$ and let
$d\sigma_\pm$ be the measure induced by the distribution  $\delta(Q\mp 1)$ on the surface $\Sigma_{\pm }$. It is well-known that
\begin{equation*}
d\xi=\sum_{\pm}\rho^{d-1}d\rho d\sigma_\pm(\theta),
\end{equation*}
where $\xi=\rho\theta$, $\rho>0$, $\theta\in\Sigma_{\pm}$.

\subsubsection{Imaginary part} First we deal with the imaginary part, which is relatively simpler. Note that
\[
\F^{-1}\Big( \frac{b\widehat{f}}{(Q+a)^2+b^2} \Big)(x) = \sum_{\pm} \int_0^\infty \int_{\Sigma_{\pm }} \frac{b\widehat{f}(\rho \theta)e^{2\pi i x\cdot \rho\theta}}{(\pm \rho^2+a)^2+b^2} d\sigma_{\pm}(\theta) \rho^{d-1}d\rho.
\]
By Minkowski's inequality, scaling, and by Proposition \ref{restriction}, it follows
\begin{align*}
\Big\| \F^{-1}\Big( \frac{b\widehat{f}}{(Q+a)^2+b^2} \Big) \Big\|_{q,\infty}
&\lesssim \sum_{\pm} \int_0^\infty \Big\| \int_{\Sigma_{\pm }} \frac{b\widehat{f}(\rho \theta)e^{2\pi i x\cdot \theta}}{(\pm \rho^2+a)^2+b^2} d\sigma_{\pm}(\theta) \Big\|_{q,\infty} \rho^{-\frac dq+d-1}d\rho \\
&\lesssim \|f\|_{p,1} \int_0^\infty \frac{|b|\rho \rho^{\frac dp-\frac dq-2}}{(\rho^2-|a|)^2+b^2} d\rho .
\end{align*}
Here we use the fact that $L^{q,\infty}$ is normable.
Hence, it is sufficient to show that
\begin{equation}\label{integral}
\int_0^\infty \frac{|b| \rho^{-\sigma}}{(\rho-|a|)^2+b^2} d\rho \le C
\end{equation}
with $C$ independent of $a, b$ when $a^2+b^2=1$, where  \[\sigma=1-\frac d{2p}+\frac d{2q}.\]
So, $0\le \sigma\le \frac{1}{d+1}.$ To show \eqref{integral}  we consider the cases $10^{-2} |a|\le  |b|\le 10^{2} |a| $, $|a|< 10^{-2} |b|$, and $|a|> 10^2 |b|$, separately.
The first two cases are easy to check.  For the last case, splitting
\[\int_0^\infty \frac{|b| \rho^{-\sigma}}{(\rho-|a|)^2+b^2} d\rho=
\Big(\int_{0}^{|a|-|b|} + \int_{|a|-|b|}^{|a|+|b|} + \int_{|a|+|b|}^\infty\Big) \frac{|b| \rho^{-\sigma}}{(\rho-|a|)^2+b^2} d\rho\]
and using $|a|\sim 1\gg |b|$, it is not difficult to see the three integrals are uniformly bounded.

\subsubsection{Real part} For the real part we  show 
\begin{equation}\label{ab}
\Big\|\F^{-1}\Big( \frac{(Q+a)\wh{f}}{(Q+a)^2+b^2} \Big) \Big\|_{q,\infty} \leq C\norm{f}_{p,1},
\end{equation}
 uniformly in $a,b \in \R$ provided $ a^2+b^2=1$.
As in \cite{KRS} by a density argument  we may assume  $b\neq 0$. In fact, the case $b=0$ in which
 \eqref{ab} is understood as $\big\|\F^{-1}(\operatorname{p. v.} {(Q\pm 1)^{-1}}{\wh{f}}\, ) \big\|_{q,\infty} \leq C\norm{f}_{p,1}$  is already handled in Section \ref{pvpv}.

We start by decomposing the multiplier.  We make  use of the particular functions $\varphi$, $\psi$ which are constructed in Lemma \ref{pv} so that
 \begin{equation*}
 1=\sum_{l=-\infty}^{\infty}\varphi (2^{-l}x)=\sum_{l=-\infty}^{\infty} 2^{-l} x\psi (2^{-l}x).
 \end{equation*}

\smallskip
Let $ l_0$ be the number such that $2^{l_0-1}< |b|\le 2^{l_0}$.   Let us set
\begin{align*}
A_l&=\frac{Q(\xi)+a}{(Q(\xi)+a)^2+b^2} \varphi (2^{-l} (Q(\xi)+a )),\\
B_l&=\Big(\frac{Q(\xi)+a}{(Q(\xi)+a)^2+b^2}- \frac{1}{Q(\xi)+a}\Big)\varphi (2^{-l} (Q(\xi)+a )),\\
C_l&= \frac{1}{Q(\xi)+a}\varphi (2^{-l} (Q(\xi)+a )).
\end{align*}
This gives a decomposition of  the multiplier as follows;
{\begin{align}\label{decompABC}
&~ \qquad \frac{Q(\xi)+a}{(Q(\xi)+a)^2+b^2}
=\sum_{l<l_0} A_l+ \sum_{l\ge l_0}  B_l+ \sum_{l\ge l_0}  C_l.
\end{align} }

Now, in order to prove \eqref{ab}   we consider the two cases $(1/p,1/q)=B$ and $(1/p,1/q)=C$, separately.\\

%\subsubsection*{Proof of \eqref{ab} for $(1/p,1/q)=B$}
\noindent\textit{Proof of \eqref{ab} for $(1/p,1/q)=B$.}
The operator which corresponds to $ \sum_{l\ge l_0}  C_l$ can be handled by the exactly same argument as in the section 3.1. 
Note that $C_l= 2^{-l} \psi( 2^{-l} (Q(\xi)+a))$ and $\supp \widehat \psi\subset \{t: 1/2\le |t|\le 2\}$
and  recall that we are assuming \eqref{supp-con}.  Hence,
one may repeat the argument in the subsections \ref{rest-b}, \ref{rest-c} by making use of  the bounds  for $f\to \finv{C_l\widehat f\,\,}$  in Proposition  \ref{global-bound} (\eqref{psi2}) and Lemma \ref{interpol}  to get, for $(1/p, 1/q)=B$,
\begin{equation}\label{sumcl}\Big\| \finv{\sum_{l\ge l_0}   C_l \widehat f\,\,}\Big\|_{q,\infty} \lesssim \|f\|_{p,1}.\end{equation}

The boundedness of multiplier operators given by  $\sum_{l<l_0} A_l$ and $\sum_{l\ge l_0}  B_l$ can be shown  by the similar argument for
the imaginary part in \eqref{re+im}.   We first handle the operator given by  $\sum_{l\ge l_0}  B_l$. Note that
\[B_l= \frac{-b^2}{(Q(\xi)+a)^2+b^2} 2^{-l}\psi(2^{-l}( Q(\xi)+a)).\]  Since  $\psi$ is bounded,
using the generalized polar coordinates and Minkowski's inequality as before,  we have
\begin{align*}
&
~ \quad \Big\| \int_{\R^d} B_l(\xi) \widehat{f}(\xi)e^{2\pi i x\cdot \xi} d\xi \Big\|_{q,\infty} \\
&
\lesssim  2^{-l} b^2\sum_{\pm} \int_0^\infty \Big\| \int_{\Sigma_{\pm}}
 \frac{\psi(2^{-l}(\pm \rho^2+a))}{(\pm \rho^2+a)^2+b^2} \widehat{f}(\rho \theta)
 e^{2\pi i x\cdot \theta} d\sigma_{\pm}(\theta) \Big\|_{q,\infty} \rho^{-\frac dq+d-1}d\rho
\\
&\lesssim 2^{-l} b^2\|f\|_{p,1}    \sum_{\pm} \int_0^\infty
\frac{ \rho\rho^{\frac dp-\frac dq-2}  }{(\pm \rho^2+a)^2+b^2} d\rho.
\end{align*}
Note that ${\frac dp-\frac dq-2} =0$ because $(1/p,1/q)=B$.  Recalling $2^l\ge |b|$ and taking the summation over $l$, we have
\begin{equation}\label{sumbl}
\Big\| \int_{\R^d} \sum_{l\ge l_0} B_l(\xi) \widehat{f}(\xi)e^{2\pi i x\cdot \xi} d\xi \Big\|_{q,\infty}
\lesssim  \|f\|_{p,1} \int_0^\infty   \frac{ |b|\rho}{( \rho^2-|a|)^2+b^2} d\rho .
\end{equation}
This gives the desired uniform bound because the last integral is bounded uniformly in $a, b$ when $a^2+b^2=1$ (see \eqref{integral}).

\medskip

Now we consider the part given by $\sum_{l<l_0} A_l$.  Similarly,
\begin{align*}
&~ \quad \Big\| \int_{\R^d}  A_l(\xi) \widehat{f}(\xi)e^{2\pi i x\cdot \xi} d\xi \Big\|_{q,\infty} \\
&
\lesssim \sum_{\pm} \int_0^\infty \Big\| \int_{\Sigma_{\pm}} \frac{ (\pm \rho^2+a)
\varphi(2^{-l}(\pm \rho^2+a))}{(\pm \rho^2+a)^2+b^2}  \widehat{f}(\rho \theta)e^{2\pi i x\cdot \theta}
d\sigma_{\pm}(\theta) \Big\|_{q,\infty} \rho^{-\frac dq+d-1}d\rho
\\
&\lesssim \|f\|_{p,1}    \sum_{\pm} \int_0^\infty  \frac{ \rho \rho^{\frac dp-\frac dq-2}  |\pm \rho^2+a| |\varphi(2^{-l}(\pm \rho^2+a))| }{(\pm \rho^2+a)^2+b^2} d\rho.
\end{align*}
Since
$ |\varphi(\rho)|\lesssim (1+\rho)^{-M},$    $\sum_{l<l_0}  |t\,\varphi(2^{-l} t)|\lesssim  \sum_{l<l_0}  2^l$. So, we have   $\sum_{l<l_0}  |t\,\varphi(2^{-l} t)|\lesssim |b|$ because $2^{l_0}\lesssim |b|$.  Now, taking summation over $l$, we get
\begin{equation}\label{sumal}
\Big\| \int_{\R^d}  \sum_{l<l_0} A_l(\xi) \widehat{f}(\xi)e^{2\pi i x\cdot \xi} d\xi \Big\|_{q,\infty} \lesssim \|f\|_{p,1}    \int_0^\infty
\frac{ |b|\rho  }{(\rho^2-|a|)^2+b^2} d\rho.\end{equation}
As before this gives the desired uniform bound by \eqref{integral}.

Therefore, combining \eqref{sumcl}, \eqref{sumbl}, and \eqref{sumal}, we get the uniform bound \eqref{ab} for $(1/p,1/q)=B$.\\

\noindent\textit{Proof of \eqref{ab} for $(1/p,1/q)=C$.}
We  distinguish  the cases $|b|< |a|$ and $|a|\le |b|$.

The case $|b|<|a|$  can be treated similarly by following the lines of argument for the former case $(1/p,1/q)=B$. We use the decomposition \eqref{decompABC} and, then,  the bounds for the multipliers given by $\sum_{l<l_0}A_l$ and $\sum_{l\ge l_0}B_l$ follow  from the same argument for the case $(1/p,1/q)=B$. So, we omit the detail.   However,  for the part given by $\sum_{l\ge l_0}   C_l$  we have
\begin{equation*}\Big\| \finv{\sum_{l\ge l_0}   C_l \widehat f\,\,}\Big\|_{q,\infty} \lesssim |a|^{-\frac{1}{2}(\frac{1}{p}-\frac{1}{q})}\|f\|_{p,1}\end{equation*}
from Proposition \ref{global} (\eqref{psi1}) and Lemma \ref{interpol}. But, since   $|b|<|a|$ ($1/\sqrt2 \le|a|\le1$), we get the uniform bound  \eqref{sumcl}  for $(1/p,1/q)=C$.  Combining all these estimates, we get  \eqref{ab}  for $(1/p,1/q)=C$ when $|b|<|a|$.

For the case $|a|\le |b|$ we don't need the decomposition \eqref{decompABC}. The multiplier operator can be handled easily by making use of Proposition \ref{restriction} and the generalized polar coordinates.
 Similarly as before,   using the polar coordinates, Minkowski's inequality, and   Proposition \ref{restriction},
\begin{align*}
 \Big\|\F^{-1}\Big( \frac{(Q+a)\wh{f}}{(Q+a)^2+b^2} \Big) \Big\|_{q,\infty}
&
\lesssim  \|f\|_{p,1}    \sum_{\pm} \int_0^\infty
\frac{ |a\pm \rho^2| \rho\rho^{\frac dp-\frac dq-2}  }{| a\pm \rho^2|^2+b^2} d\rho
\\
&
\lesssim  \|f\|_{p,1}    \int_0^\infty
\frac{ |\rho-|a|| \rho^{-\frac{1}{d+1}} }{|  \rho-|a||^2+b^2} d\rho\,.
\end{align*}
Since $|a|\le 1/\sqrt2 \le |b| $, by splitting the integral $\int_0^\infty=\int_0^1+\int_1^\infty$ it is easy to see the last integral
is uniformly bounded if $a^2+b^2=1$ and   $|a|\le |b|$.  Hence,  we get \eqref{ab}  for $(1/p,1/q)=C$ when $|a|\le |b|$. This completes the proof of Theorem \ref{uniSob1}.

\section{Application to unique continuation problem for non-elliptic equation operators}

As an immediate consequence of  the non-elliptic uniform Sobolev inequality \eqref{Sob} in Theorem \ref{mainthm},
we  have  a type of Carleman  estimates \eqref{cman}. As their applications, one also obtains results on unique continuation. For the elliptic case, although only the dual case ($1/p+1/q=1$, $1/p-1/q=2/d$) is explicitly stated  in \cite{KRS} (pp. 342--346), the corresponding statements for any $p, q$ are true as long as  the  uniform Sobolev estimate (\cite{KRS})  holds. Likewise, the enlarged range of $p,q$ for which  the  uniform Sobolev inequalities  for
non-elliptic operators  hold extends the class of functions for which unique continuation holds.
What follows  can be proved by routine adaptation of the argument in \cite{KRS} once we have the uniform Sobolev inequality \eqref{Sob}. So, we state our results without giving proofs.

\begin{cor}\label{Carleman}
Let $d\ge 3$ and let $P(D)= Q(D)+\sum_{j=1}^d a_j D_j + b$ with the non-elliptic principal symbol $Q$ as in \eqref{quad}, where $a_j, b \in \C$. Suppose $p$, $q$ satisfy $1/p-1/q=2/d$ and  $2d(d-1)/(d^2+2d-4)<p<2(d-1)/d$,  then we have
\begin{equation}\label{cman}
\norm{e^{tv\cdot x}u}_{L^q(\R^d)} \le C \|e^{tv\cdot x}P(D)u\|_{L^p(\R^d)}, \quad u\in \cS(\mathbb{R}^d),
\end{equation}
where the constant $C$ is independent of $t\in\R$, $v\in\R^d$, and $a_j, b$.
\end{cor}

Consequently, this extends the class of functions  for which the global and local unique continuation properties for the differential inequality $|P(D)u|\le |Vu|$ hold.

\begin{cor}\label{globaluc}
Let $d\ge3$ and let $p$ and $P(D)$ be as in Corollary \ref{Carleman}. Suppose $V\in L^{d/2}(\mathbb{R}^d)$. If the support of $u\in W^{2,p}(\mathbb{R}^d)$ is contained in a half space and $u$ satisfies $|P(D)u|\le|Vu|$ almost everywhere, then $u=0$ on the whole $\R^d$.
\end{cor}

\begin{prop}
Let $d\ge3$, $2d(d-1)/(d^2+2d-4)<p<2(d-1)/d$,  and let $P(D)= \bigboxvoid +\sum_{j=1}^d a_jD_j+b$, where $\bigboxvoid=D^2_1 -\sum^d_{j=2}D^2_j$ is the wave operator and  $a_j,b\in\C$. Suppose $\mathcal{C}_v$ is an open convex cone with vertex $v\in \mathbb R^d$ such that every characteristic hyperplane with respect to $P(D)$ through $v$ intersects $\overline{\mathcal{C}_v}\setminus \{v\}$. If $u\in W_{loc}^{2,p}(\mathcal{C}_v)$ and $V\in L_{loc}^{d/2}(\mathcal{C}_v)$ satisfy the differential inequality
$
|P(D)u|\le|Vu|
$
in $\mathcal{C}_v$,  then $u=0$ in the whole $\mathcal{C}_v$ whenever $u$ vanishes outside a bounded subset of $\mathcal{C}_v$.
\end{prop}

\medskip

\subsection*{Acknowledgement}  This work was supported by NRF of  Republic of Korea (grant No. 2015R1A2A2A05000956). The authors would like to thank J.-G. Bak for 
communication regarding earlier results. 

 {\bibliographystyle{plain}

\end{document}